\documentclass[12pt,paper=a4, headsepline,headings=normal,abstract=true]{scrartcl} 
\usepackage[english]{babel}
\usepackage[utf8]{inputenc}
\usepackage[T1]{fontenc} 
\usepackage{textcomp}
\usepackage[margin=2cm]{geometry}
\usepackage[reqno]{amsmath}   
\usepackage{amssymb}          
\usepackage{amsthm}           
\usepackage{amstext}          
\usepackage{enumerate}        
\usepackage{bbm}
\usepackage{pifont}           
\usepackage{framed}
\usepackage{scrpage}
\usepackage{graphicx}
\usepackage{xcolor}
\usepackage[final]{changes}
\usepackage{verbatim}
\usepackage{mathrsfs}
\usepackage{url}
\urldef{\urluni}{\url}{http://www.mathematik.uni-kl.de/fuana/}
\urldef{\emailfattler}{\url}{fattler@mathematik.uni-kl.de}
\urldef{\emailgrothaus}{\url}{grothaus@mathematik.uni-kl.de}
\urldef{\emailvosshall}{\url}{vosshall@mathematik.uni-kl.de}

\setkomafont{pageheadfoot}{%
\normalfont\normalcolor\itshape\small
}
\setkomafont{pagenumber}{\normalfont\bfseries}

\def\stackunder#1#2{\mathrel{\mathop{#2}\limits_{#1}}}

\makeatletter\@addtoreset{equation}{section}\makeatother

\allowdisplaybreaks

\theoremstyle{plain}      \newtheorem{theorem}{Theorem}[section]
                          \newtheorem{corollary}[theorem]{Corollary}
                          \newtheorem{proposition}[theorem]{Proposition}
													\newtheorem{condition}[theorem]{Condition}

\theoremstyle{remark}     \newtheorem{remark}[theorem]{Remark}
                          \newtheorem{lemma}[theorem]{Lemma}

\theoremstyle{definition} \newtheorem{definition}[theorem]{Definition}

\begin{document} 

\newcommand{\grad}{\nabla}
\newcommand{\D}{\partial}
\newcommand{\E}{\mathcal{E}}
\newcommand{\N}{\mathbb{N}}
\newcommand{\R}{\mathbb{R}_{\scriptscriptstyle{\ge 0}}}
\newcommand{\dom}{\mathcal{D}}
\newcommand{\ess}{\operatorname{ess~inf}}
\newcommand{\cem}{\operatorname{\text{\ding{61}}}}
\newcommand{\supp}{\operatorname{\text{supp}}}
\newcommand{\ca}{\operatorname{\text{cap}}}

\begin{titlepage}
\title{\Large Construction and analysis of a sticky reflected distorted Brownian motion}
\author{\normalsize\sc Torben Fattler\footnote{University of Kaiserslautern, P.O.Box 3049, 67653
Kaiserslautern, Germany.}~\thanks{\emailfattler}~\footnotemark[5]\and \normalsize\sc Martin Grothaus\footnotemark[1]~\thanks{\emailgrothaus}~\footnotemark[5] 
 \and \normalsize\sc Robert Vo\ss hall\footnotemark[1]~\thanks{\emailvosshall}~\thanks{\urluni}}
\date{\small\today}
\end{titlepage}
\maketitle

\pagestyle{headings}

\begin{abstract}
\added{We give a Dirichlet form approach for the construction of a distorted Brownian motion in $E:=[0,\infty)^n$, $n\in\mathbb{N}$, where the behavior on the boundary is determined by the competing effects of reflection from and pinning at the boundary (sticky boundary behavior). In providing a Skorokhod decomposition of the constructed process we are able to justify that the stochastic process is solving the underlying stochastic differential equation weakly for quasi every starting point with respect to the associated Dirichlet form. That the boundary behavior of the constructed process indeed is sticky, we obtain by proving ergodicity of the constructed process. Therefore, we are able to show that the occupation time on specified parts of the boundary is positive. 
In particular, our considerations enable us to construct a dynamical wetting model (also known as Ginzburg--Landau dynamics) on a bounded set $D_{\scriptscriptstyle{N}}\subset \mathbb{Z}^d$ under mild assumptions on the underlying pair interaction potential in all dimensions $d\in\mathbb{N}$. In dimension $d=2$ this model describes the motion of an interface resulting from wetting of a solid surface by a fluid.}
\deleted{We give a Dirichlet form approach for the construction of a distorted Brownian motion in $E:=[0,\infty)^n$, $n\in\mathbb{N}$, where the behavior on the boundary is determined by the competing effects of reflection from and pinning at the boundary. The problem is formulated in an $L^2$-setting with underlying measure $\mu=\varrho\,m$. Here $\varrho$ is a positive density, integrable with respect to the measure $m$ and fulfilling the Hamza condition. The measure $m$ is such that the boundary $\partial E$ of $E$ is not of $m$-measure zero. A reference measure $\mu$ of this type is needed in order to give meaning to the so-called Wentzell boundary condition which is in literature typical for modeling such kind of sticky boundary behavior. In view of our application we construct a dynamics even on the boundary. In providing a Skorokhod decomposition of the constructed process we are able to justify that the stochastic process is solving the underlying stochastic differential equation weakly for quasi every starting point with respect to the associated Dirichlet form. That the boundary behavior of the constructed process indeed is sticky, we obtain by proving ergodicity of the constructed process. Therefore, we are able to show that the occupation time on specified parts of the boundary is positive. In order to obtain the Skorokhod decomposition we need $\varrho$ to be continuously differentiable on $E$, which is equivalent to continuity of the logarithmic derivative $\frac{\nabla\varrho}{\varrho}$ of $\varrho$. Furthermore, we assume that $\frac{\nabla\varrho}{\varrho}$ is square integrable with respect to $\mu$. We do not need that the logarithmic derivative of $\varrho$ is Lipschitz continuous.
In particular, our considerations enable us to construct a dynamical wetting model (also known as Ginzburg--Landau dynamics) on a bounded set $D_{\scriptscriptstyle{N}}\subset \mathbb{Z}^d$ under mild assumptions on the underlying pair interaction potentials in all dimensions $d\in\mathbb{N}$. In dimension $d=2$ this model describes the motion of an interface resulting from wetting of a solid surface by a fluid.}
\\\\
\\
\end{abstract}

\section{Introduction}
In \cite{EP12} the authors study stochastic differential equations (SDEs) with sticky boundary behavior and provide existence and uniqueness of solutions to the SDE system
\begin{align}\label{equEP12}
\left\{
\begin{array}{l}
dX_t=\frac{1}{2}d{\ell}_t^{0\added{+}}\big(X\big)+\mathbbm{1}_{(0,\infty)}\big(X_t\big)\,dB_t\\
\mathbbm{1}_{\{0\}}\added{\big(X_t\big)}\,dt=\frac{1}{\mu}\,d\ell_t^{0\added{+}}\big(X\big),
\end{array}
\right.
\end{align}
for reflecting Brownian motion $X$ in $[0,\infty)$ sticky at $0$, where $X:=\big(X_t\big)_{t\ge 0}$ starts at $x\in [0,\infty)$, $\mu\in (0,\infty)$ is a given constant, $\ell^{0\added{+}}\big(X\big)$ is the \added{right} local time of $X$ at $0$ and $B:=\big(B_t\big)_{t\ge 0}$ is the standard Brownian motion. In particular, H.-J.~Engelbert and G.~Peskir show that the system (\ref{equEP12}) has a jointly unique weak solution and moreover, they prove that the system (\ref{equEP12}) has no strong solution, thus verifying Skorokhod's conjecture of the non-existence of a strong solution in this case. For an outline of the historical evolution in the study of sticky Brownian motion we refer to the references given in \cite{EP12}.

In the present paper we construct a reflected distorted Brownian motion in $E:=[0,\infty)^n$, $n\in\mathbb{N}$, with sticky boundary behavior. First we use Dirichlet form techniques in order to construct solutions in the sense of the associated martingale problem for general Wentzell type boundary conditions. Then, by providing a Skorokhod decomposition for the constructed process, we can show that this process solves the stochastic differential equation 
\begin{multline}\label{repintro}
d\mathbf{X}_{t}^j=\mathbbm{1}_{\mathring{E}}\big(\mathbf{X}_t\big)\,\sqrt{2}\,dB^j_t+\partial_j\ln(\varrho)\big(\mathbf{X}_{\replaced{t}{\tau}}\big)\mathbbm{1}_{\mathring{E}}\big(\mathbf{X}_{\replaced{t}{\tau}}\big)\,dt\\
+\sum_{\varnothing\not=B\subsetneq I}\left\{\begin{array}{ll}
  \mathbbm{1}_{{E_+(B)}}\big(\mathbf{X}_t\big)\,\sqrt{2}\,dB^j_t+\partial_j\ln(\varrho)\big(\mathbf{X}_{\replaced{t}{\tau}}\big)\mathbbm{1}_{{E_+(B)}}\big(\mathbf{X}_{\replaced{t}{\tau}}\big)\,dt, & \text{if }j\in B\\
  \frac{1}{\replaced{\beta}{s}}\,\mathbbm{1}_{{E_+(B)}}\big(\mathbf{X}_{\replaced{t}{\tau}}\big)\,dt, & \text{if }j\in I\setminus B
  \end{array}\right.\\
+\frac{1}{\replaced{\beta}{s}}\,\mathbbm{1}_{\{(0,\ldots,0)\}}\big(\mathbf{X}_{\replaced{t}{\tau}}\big)\,dt\added{,\quad\text{for some }\beta>0,}
\end{multline}
weakly for quasi every starting point with respect to the underlying Dirichlet form. Here $j\in I:=\{1,\ldots,n\}$, \deleted{$B\subsetneq I$,} $E_+(B):=\big\{x\in E\,|\,x_i>0\text{ for all }i\in B\added{\text{ and }x_i=0\text{ for all }i\in I\setminus B}\big\}$ \added{for $B \subset I$ with $E_+(B) \subset \partial E$ for $B\subsetneq I$}, $(B^j_t)_{t\ge 0}$ are one dimensional independent standard Brownian motions, $j\in I$\replaced{.}{, and} $\varrho$ is a \replaced{continuously differentiable density on $E$ such that for all $B \subset I$, $\varrho$ is almost everywhere positive on $E_+(B)$ with respect to the Lebesgue measure and for all $\varnothing\not=B\subset I$,}{positive, continuously differentiable density on $E$ such that} $\sqrt{\varrho\added{\big|_{E_{+}(B)}}}$ is in the Sobolev space of weakly differentiable functions on \replaced{$E_{+}(B)$}{$\mathring{E}$}, square integrable together with \replaced{its}{their} derivative. $\varrho$ continuously differentiable on $E$ \replaced{implies that}{is equivalent to} the drift part $\big(\partial_j\ln(\varrho)\big)_{j\in I}$ \replaced{is}{being} continuous \added{on $\{ \varrho >0 \}$}. The stochastic differential equation (\ref{repintro}) can be rewritten as  
\begin{align}\label{repintro2}
d\mathbf{X}_{t}^j=\mathbbm{1}_{(0,\infty)}\big(\mathbf{X}^j_t\big)\,\Big(\sqrt{2}\,dB^j_t+\partial_j\ln(\varrho)\big(\mathbf{X}_t\big)\,dt\Big)+\frac{1}{\replaced{\beta}{s}}\,\mathbbm{1}_{\{0\}}\big(\mathbf{X}^j_t\big)\,dt,\quad j\in I,\added{\quad\text{for some }\beta>0,}
\end{align}
or equivalently
\begin{multline*}
d\mathbf{X}_{t}^j=\mathbbm{1}_{(0,\infty)}\big(\mathbf{X}^j_t\big)\,\Big(\sqrt{2}\,dB^j_t+\partial_j\ln(\varrho)\big(\mathbf{X}_t\big)\,dt\Big)+d\ell_{\replaced{t}{0}}^{\added{0,}j},\\
\text{with}\quad \ell_{\replaced{t}{0}}^{\added{0,}j}:=\frac{1}{\replaced{\beta}{s}}\int_0^t\mathbbm{1}_{\{0\}}\big(\mathbf{X}_{\replaced{s}{\tau}}^j\big)\,d\replaced{s}{\tau},\quad j\in I\added{,\quad\text{for some }\beta>0}.
\end{multline*}

\added{Note that a solution to (\ref{repintro2}) is a continuous semimartingale. By \cite[Chapter VI]{RY91} the right local time $\ell_t^{0+,j}$ of $\big(\mathbf{X}_{t}^j\big)_{t \geq 0}$, $j \in I$, is charaterized by}

\[ \big| \mathbf{X}_{t}^j \big| = \big| \mathbf{X}_{0}^j \big| + \int_0^t \text{sgn}(\mathbf{X}_{s}^j) \, d\mathbf{X}_{s}^j + \ell_t^{0+,j}, \]

\added{where $\text{sgn}$ is defined by $\text{sgn}(x)=1$ for $x >0$ and $\text{sgn}(x)=-1$ for $x \leq 0$. For a solution to (\ref{repintro2}) holds}

\begin{align*}
\big| \mathbf{X}_{t}^j \big| &= \big| \mathbf{X}_{0}^j \big| + \int_0^t \mathbbm{1}_{(0,\infty)}\big(\mathbf{X}^j_{s} \big) \, d\mathbf{X}^j_{s} + \frac{1}{\beta}\int_0^t\mathbbm{1}_{\{0\}}\big(\mathbf{X}_{s}^j\big)\,ds \\
&= \big| \mathbf{X}_{0}^j \big| + \int_0^t \text{sgn}\big(\mathbf{X}^j_{s} \big) \, d\mathbf{X}^j_{s} + \frac{2}{\beta}\int_0^t\mathbbm{1}_{\{0\}}\big(\mathbf{X}_{s}^j\big)\,ds,
\end{align*}

\added{since $\mathbf{X}_{t}^j \geq 0$ for all $t \geq 0$ and $\text{sgn}(0)=-1$. Hence,  $\ell_t^{0+,j}=\frac{2}{\beta}\int_0^t\mathbbm{1}_{\{0\}}\big(\mathbf{X}_{s}^j\big)\,ds = 2~ \ell^{0,j}_t$ almost surely. In other words, $\ell^{0,j}_t$ equals one half of the right local time $\ell_t^{0+,j}$. Furthermore, due to \cite[Corollary 1.9]{RY91} we can conclude that $\ell^{0,j}_t$ coincides with the central local time of $\big(\mathbf{X}_{t}^j\big)_{t \geq 0}$, i.e., it holds almost surely}

\begin{align*}
\ell^{0,j}_t=\frac{1}{\beta}\int_0^t\mathbbm{1}_{\{0\}}\big(\mathbf{X}_{s}^j\big)\,ds 
&= \frac{1}{2}~ \lim_{\varepsilon \downarrow 0} \frac{1}{\varepsilon} \int_0^t \mathbbm{1}_{[0,\varepsilon)} \big(\mathbf{X}_{s}^j\big)\, d\langle \mathbf{X}^j \rangle_{s} \\
&= \lim_{\varepsilon \downarrow 0} \frac{1}{2 \varepsilon} \int_0^t \mathbbm{1}_{(-\varepsilon,\varepsilon)} \big(\mathbf{X}_{s}^j\big)\, d\langle \mathbf{X}^j \rangle_{s}.
\end{align*}

Our considerations are motivated by the so-called $\nabla\phi$ interface model which provides a fundamental mathematical model for the physical description of interfaces from a microscopic or mesoscopic point of view. As an application of our results we are interested in the time development of such interfaces. In \cite{FuSpo97} the authors consider a scalar field $\boldsymbol{\phi}_t$, $t\ge 0$, where its motion is governed by a reversible stochastic dynamics\replaced{, i.e.,}{. I.e.} in a finite volume $\Lambda\subset \mathbb{Z}^d$, $d\in\mathbb{N}$, under suitable boundary conditions, the scalar field \added{$\boldsymbol{\phi}_t:=\big(\boldsymbol{\phi}_t(x)\big)_{x\in\Lambda}$, $t\ge 0$}, is described by the stochastic differential equations
\begin{align*}
d\boldsymbol{\phi}_t(x)=-\sum_{\stackunder{\scriptscriptstyle{|x-y|=1}}{\scriptscriptstyle{y\in\Lambda}}}V'(\boldsymbol{\phi}_t(x)-\boldsymbol{\phi}_t(y))dt+\sqrt{2}\,dB_t(x),\quad x\in\Lambda,\quad t\ge 0.
\end{align*}
Here $|\cdot|\deleted{_{\scriptscriptstyle{\text{euc}}}}$ denotes the norm induced by the euclidean scalar product on $\mathbb{R}^d$, $V\in C^2(\mathbb{R})$ is a symmetric, strictly convex potential and $\big\{(B_t(x))_{t\ge 0}\,|\,x\in\Lambda\big\}$ are independent standard Brownian motions. Such a dynamics is known as the \emph{Ginzburg-Landau $\nabla\phi$ interface model in finite volume}.
Of particular interest in the framework of $\nabla\phi$ interface models is the so-called \emph{entropic repulsion}. Though one considers  the $\nabla\phi$ interface model with reflection on a hard wall. This phenomenon was investigated e.g.~in \cite{DeuGia00} and \cite{BDG01} for the static $\nabla\phi$ interface model. Interface motion with entropic repulsion, i.e., the Ginzburg-Landau $\nabla\phi$ interface model with entropic repulsion was studied recently in \cite{DeuNis07} for dimension $d\ge 2$. Here the underlying potentials are again symmetric, strictly convex and nearest neighbor $C^2$-pair potentials. The Ginzburg-Landau dynamics with repulsion was introduced by T.~Funaki and S.~Olla in \cite{Fu03, FuOl01}. In \cite{Za04} this problem was tackled via Dirichlet form techniques in dimension $d=1$.

In considering the $\nabla\phi$ interface model with reflection on a hard wall and additionally putting a pinning effect on that wall, we are dealing with the so-called \emph{wetting model}. In dimension $d=2$ this model describes the wetting of a solid surface by a fluid. The static wetting model was studied recently in \cite{DGZ05}, see also \cite{CaVe00}. Considerations of the Ginzburg-Landau dynamics with reflection on a hard wall under the influence of an outer force, causing e.g.~a mild pinning effect on the wall can be found in \cite{Fu03}. 

In \cite[Sect.~15.1]{Fu05} J.-D. Deuschel and T. Funaki investigated the scalar field $\boldsymbol{\phi}_t:=\big(\boldsymbol{\phi}_t(x)\big)_{x\in\Lambda}$, $t\ge 0$, described by the stochastic differential equations
\begin{multline}\label{sde}
d\boldsymbol{\phi}_t(x)=-\mathbbm{1}_{(0,\infty)}\big(\boldsymbol{\phi}_t(x)\big)\sum_{\stackunder{\scriptscriptstyle{|x-y|=1}}{\scriptscriptstyle{y\in\Lambda}}}V'\big(\boldsymbol{\phi}_t(x)-\boldsymbol{\phi}_t(y)\big)\,dt\\
+\mathbbm{1}_{(0,\infty)}\big(\boldsymbol{\phi}_t(x)\big)\sqrt{2}dB_t(x)+d\ell_{t}^{\scriptscriptstyle{0}}(x),\quad x\in\Lambda,
\end{multline}
subject to the conditions:
\begin{align*}
&\boldsymbol{\phi}_t(x)\ge 0,\quad \ell_{t}^{\scriptscriptstyle{0}}(x)\mbox{ is non-decreasing with respect to }t,\quad \ell^{\scriptscriptstyle{0}}_{0}(x)=0,\\
&\int_0^\infty\boldsymbol{\phi}_t(x)\,d\ell_{t}^{\scriptscriptstyle{0}}(x)=0,\\
&\beta\ell_{t}^{\scriptscriptstyle{0}}(x)=\int_0^t\mathbbm{1}_{\{0\}}\big(\boldsymbol{\phi}_{\replaced{s}{\tau}}(x)\big)\,d\replaced{s}{\tau}\quad\mbox{for fixed }\beta> 0,\nonumber\\
\end{align*}
where $\ell_{t}^{\scriptscriptstyle{0}}(x)$ denotes the \emph{\added{central} local time} of $\boldsymbol{\phi}_t(x)\mbox{ at }0$ and the pair interaction potential $V$ is again symmetric, strictly convex and $C^2$.

For treating this system of stochastic differential equations the authors gave reference to classical solution techniques as developed e.g.~in \cite{WaIk89}. The methods provided therein require more restrictive assumptions on the drift part as in our situation \replaced{(instead of boundedness and Lipschitz continuity we only need continuity and a mild integrability condition, see Condition \ref{condweakdiff} and Remark \ref{remcondequi})}{(we only need that the drift part is continuous, not necessarily Lipschitz continuous)}, moreover, do not apply directly (the geometry differs). First steps in the direction of applying \cite{WaIk89} are discussed in \cite{Fu05} by J.-D. Deuschel and T. Funaki. 

\added{As far as we know the only reference that applies to the system of stochastic differential equations (\ref{repintro2}) is \cite{Gra88}. By means of a suitable choice of the coefficients the system of equations given by \cite[(II.1)]{Gra88} coincides with (\ref{repintro2}), but amongst others the drift part is also assumed to be Lipschitz continuous and boundend. For this reason, it is not possible to apply the results of \cite{Gra88} to the setting invenstigated by J.-D.~Deuschel and T.~Funaki, since the potential $V$ naturally causes an unbounded drift. Moreover, neither properties of the corresponding $L^2$-semigroup are worked out nor the invariant measure, Dirichlet form or generator are provided. Such tools are very useful for analyzing scaling limits of the considered system, see e.g.~\cite{GKLR01} and \cite{Za04}. These we plan to investigate in a follow-up article.}

\added{The theory of Dirichlet forms provides appropiate techniques in order to construct and analyze solutions to (\ref{sde}) for a large class of potentials.}
\replaced{Indeed,}{Hence} we obtain a weak solution to (\ref{sde}) \deleted{below} with sticky boundary behavior under rather mild assumption on the underlying probability density. Note that in view of the results provided in \cite{EP12}, this notion of solution is the only reasonable one. That the boundary behavior of the constructed weak solution to (\ref{sde}) indeed is sticky, we obtain by proving an ergodicity result (see Theorem \ref{spendtime} below). From this we can conclude, that the occupation time on the boundary of the constructed process increases asymptotically linear, whenever the process starts in a point with positive density $\varrho$ connected with the boundary, see Corollary \ref{corospendtime} below.

\deleted{As far as we know our considerations are the first with regard to construction and analysis of the dynamical wetting model in finite volume. Our approach is based on Dirichlet form techniques. Since we obtain a one to one correspondence to the 
ional analytic tool, as Dirichlet form and operator semigroup, we expect this method in future to be useful in view of studying scaling limits for the underlying system of stochastic differential equations as done in e.g.~\cite{Za04}.}

\added{A Skorokhod decomposition for reflected diffusions on bounded Lipschitz domains with singular non-reflection part was provided by G. Trutnau in \cite{Tru03}. Here we consider the case of the Wentzell type boundary condition.} 
Dirichlet form methods in the context of Wentzell boundary condition were introduced in e.g.~\cite{VoVo03}. Here, however, in view of our application we construct via the underlying bilinear form a dynamics even on the boundary. In \cite{VoVo03} a static boundary behavior is realized. An overview of the state of the art in the framework of interface models is presented in e.g.~\cite{Ga02}, \cite{Fu05}.

Our paper is organized as follows. In Section \ref{sectfuana} we provide the functional analytic\deleted{al} background to apply Dirichlet form methods in order to tackle the problem of sticky reflected distorted Brownian motion. We analyze the bilinear form (\ref{form1}) below and show in Theorem \ref{theosumdiri} and in the proof of Lemma \ref{lemcap} that $\big(\mathcal{E},D(\mathcal{E})\big)$ is a \added{recurrent, hence in particular} conservative, strongly local, strongly regular, symmetric Dirichlet form on the underlying $L^2$-space. In Section \ref{sectprocess} we present the probabilistic counterpart of Section \ref{sectfuana}. The main result of this section is obtained in Theorem \ref{theoprocess}, where we show that $\big(\mathcal{E},D(\mathcal{E})\big)$ has an associated conservative diffusion process $\mathbf{M}$, i.e., an associated strong Markov process with continuous sample paths and infinite life time. The diffusion process $\mathbf{M}$ is analyzed in Section \ref{sectanapro}. Here we provide in Corollary \ref{coroskoro} a Skorokhod decomposition of $\mathbf{M}$. This proves that $\mathbf{M}$ is a weak solution to (\ref{sde}). In Section \ref{sectergo} we show in Theorem \ref{spendtime} that the constructed process $\mathbf{M}$ is ergodic. Moreover, we present the consequences of the ergodicity result for the occupation time on the boundary of the constructed process, see Corollary \ref{corospendtime} below. Finally, we apply our results to the problem of the dynamical wetting model, see Theorems \ref{theosumdiriapp}, \ref{theoprocessapp}, \ref{theomartingaleapp} and Corollary \ref{coroskoroapp} below.

The following list of main results summarizes the progress achieved in this paper:
\begin{enumerate}
\item[(i)]
We construct conservative diffusion processes in $[0,\infty)^n$, $n\in\mathbb{N}$, with the competing effects of reflection and pinning at the boundary (sticky reflected distorted Brownian motion) under mild assumptions on the drift part, see Theorems \ref{theoprocess} and \ref{theomartingale} below.
\item[(ii)]
We provide a Skorokhod decomposition of the constructed processes and thereby prove that the processes solve the underlying stochastic differential equations weakly for \deleted{$\mathcal{\E}$-}quasi all starting points, see Corollary \ref{coroskoro} below. 
\item[(iii)]
We show ergodicity of the constructed processes, see Theorem \ref{spendtime} below. Using this ergodicity result, we illustrate the behavior of the processes at the boundary by studying the occupation times on specified parts of the boundary by the constructed processes, see Corollary \ref{corospendtime} below.
\item[(iv)]
Our general considerations apply to the construction of the dynamical wetting model in finite volume and all dimensions $d\in\mathbb{N}$ for a large class of pair interaction potentials, see Theorems \ref{theoprocessapp}, \ref{theomartingaleapp} and Corollary \ref{coroskoroapp} below.
\end{enumerate}

\section{The functional \replaced{analytic}{analytical} background}\label{sectfuana}
Let $n\in\mathbb{N}$, $I:=I_n:=\big\{1,\ldots,n\big\}$ and $E:=E_n:=[0,\infty)^n$. We have that $\mathring{E}=(0,\infty)^n$ and we denote by $\partial E$ the boundary of $E$. 
For each $x\added{=(x_1,\ldots,x_n)}\in{E}$ we set
\begin{align*}
I_{\scriptscriptstyle{0}}(x):=\big\{i\in I\,\big|\,x_i=0\big\}\quad\mbox{and}\quad I_{\scriptscriptstyle{+}}(x):=\big\{i\in I\,\big|\,x_i>0\big\},
\end{align*}
and define for $A,B\subset I$,
\begin{align*}
{E}_{\scriptscriptstyle{0}}(A):=\Big\{x\in E\,\Big|\,I_{\scriptscriptstyle{0}}(x)=A\Big\}\quad\mbox{and}\quad {E}_{\scriptscriptstyle{+}}(B):=\Big\{x\in E\,\Big|\,I_{\scriptscriptstyle{+}}(x)=B\Big\},
\end{align*}
respectively.
\begin{remark}
We have the decomposition
\begin{align*}
{E}=\dot{\bigcup}_{A\subset I}{E}_{\scriptscriptstyle{0}}(A)=\dot{\bigcup}_{B\subset I}{E}_{\scriptscriptstyle{+}}(B). 
\end{align*}
In particular,
\begin{align*}
\partial E={E}\setminus \mathring{E}=\dot{\bigcup}_{\varnothing\not=A\subset I}{E}_{\scriptscriptstyle{0}}(A)=\dot{\bigcup}_{B\subsetneq I}{E}_{\scriptscriptstyle{+}}(B).
\end{align*}
\end{remark}
On $\big({E},\mathcal{B}\replaced{(E)}{_{\scriptscriptstyle{{E}}}}\big)$ with $\mathcal{B}\replaced{(E)}{_{\scriptscriptstyle{{E}}}}$ being the trace $\sigma$-algebra of the Borel $\sigma$-algebra $\mathcal{B}(\mathbb{R}^n)$ on ${E}$ we define for fixed $\replaced{\beta}{s}\in(0,\infty)$ the measures 
\begin{align}\label{defmeasure}
m_{n,\replaced{\beta}{s}}:=\sum_{B\subset I}\lambda_{\scriptscriptstyle{B}}^{n,\replaced{\beta}{s}}\quad\text{with}\quad\lambda_{\scriptscriptstyle{B}}^{n,\replaced{\beta}{s}}:=\replaced{\beta}{s}^{n-\# B}\,\lambda_{\scriptscriptstyle{B}}^{\scriptscriptstyle{(n)}}\quad\text{and}\quad\lambda_{\scriptscriptstyle{B}}^{\scriptscriptstyle{(n)}}:=\prod_{i\in B}dx_{\scriptscriptstyle{+}}^i\prod_{j\in I\setminus B}d\delta^j_{\scriptscriptstyle{0}},
\end{align}
where $\#S$ \replaced{denotes}{gives} the number of elements in a set $S$, $dx^i_{\scriptscriptstyle{+}}$ is the Lebesgue measure on\\ $\big([0,\infty),\mathcal{B}\big([0,\infty)\big)\big)$ and $\delta_{\scriptscriptstyle{0}}^j$ denotes the Dirac measure on $\big([0,\infty),\mathcal{B}\big([0,\infty)\big)\big)$ at $0$. The indices $i,j\in I$ give reference to the component of $x=(x_1,\ldots,x_n)\in E$ being integrated by $dx^i_{\scriptscriptstyle{+}}$ and $\delta^j_{\scriptscriptstyle{0}}$, respectively. 

\begin{condition}\label{conddensity}
$\varrho$ is a $m_{n,\replaced{\beta}{s}}$-a.e. positive function on ${E}$ such that $\varrho\in L^{1}\big({E};m_{n,\replaced{\beta}{s}}\big)$.
\end{condition}

\begin{remark}\label{remprobdensity}
In particular, $\varrho$ can be chosen to be a probability density.
\end{remark}
Under Condition \ref{conddensity} we define on $\big(E,\mathcal{B}\replaced{(E)}{_{\scriptscriptstyle{E}}}\big)$ the measure $\mu_{n,\replaced{\beta}{s},\varrho}:=\varrho\,m_{n,\replaced{\beta}{s}}$ and hence, the space of square integrable functions on $E$ with respect to $\mu_{n,\replaced{\beta}{s},\varrho}$, denoted by $L^2\big(E;\mu_{n,\replaced{\beta}{s},\varrho}\big)$. 

\begin{remark}\label{remBaire}
Note that the measure $\mu_{n,\replaced{\beta}{s},\varrho}$ on $\big(E,\mathcal{B}\replaced{(E)}{_{\scriptscriptstyle{E}}}\big)$ is a \emph{Baire measure}. In our setting this means $\mu_{n,\replaced{\beta}{s},\varrho}$ is a Borel measure with the additional property that
\begin{align}\label{Baire}
\mu_{n,\replaced{\beta}{s},\varrho}(K)<\infty\quad\text{for all compact sets}\quad K\subset E.
\end{align}
(\ref{Baire}) is fulfilled, since $\varrho\in L^{1}\big(E;m_{n,\replaced{\beta}{s},\varrho}\big)$. Obviously, $E$ is locally compact and countable at infinity.
\end{remark}

We set
\begin{align*}
C^{{0}}_{{c}}\big(E\big):=\Big\{f:E\to\mathbb{R}\,\Big|\,f\mbox{ is continuous on }{E}
\mbox{ with }\supp(f)\subset E\mbox{ compact}\Big\},
\end{align*}
where $\supp$ denotes the support of the corresponding function and for $k\in\mathbb{N}$ we define
\begin{multline*}
C^{{k}}_{{c}}\big(E\big):=\Big\{f:E\to\mathbb{R}\,\Big|\,f\mbox{ is $k$-times continuously differentiable on }\mathring{E}\\
\mbox{ with }\supp(f)\subset E\mbox{ compact}\mbox{ and }\\
\partial^lf\mbox{ extends continuously to }E\mbox{ for $|l|\le k$}\Big\}.
\end{multline*}
Here and below $\partial^lf$ denotes the partial derivative of $f$ to the multi index $l\in\mathbb{N}^{{n}}_{\scriptscriptstyle{0}}$, i.e., 
\begin{multline*}
l=\big(l_1,\ldots,l_n\big)\in\mathbb{N}^{{n}}_{\scriptscriptstyle{0}},\quad |l|=l_1+\ldots+l_n,\\
\partial^lf:=\partial_1^{l_1}\partial_2^{l_2}\ldots\partial_n^{l_n}f,\quad\partial_i^{l_i}f:=\partial_{x_i}^{l_i}f,\quad \partial_{x_i}^0f:=f,\quad i\in I.
\end{multline*}
We write $\partial_i$ instead of $\partial^1_i$. Furthermore, $C_c^\infty({E}):=\bigcap_{k\in\mathbb{N}_{\scriptscriptstyle{0}}}C^k_c(E)$.

\begin{remark}\label{propdensedef}
Under Condition \ref{conddensity} we have that $C_c^\infty\big(E\big)$ is dense in $L^2\big(E;\mu_{n,\replaced{\beta}{s},\varrho}\big)$.
\end{remark}

\deleted{proof}

\subsection{Dirichlet forms}
Let $n\in\mathbb{N}$ be fixed and denote by $\replaced{\{}{(}e_1, \dots, e_n\replaced{\}}{)}$ the canonical basis of $\mathbb{R}^n$. For $\replaced{\beta}{s}\in(0,\infty)$ and $\varrho$ fulfilling Condition \ref{conddensity} we define on $L^2\big(E;\mu_{n,\replaced{\beta}{s},\varrho}\big)$ the bilinear form
\begin{align}\label{form1}
\mathcal{E}(f,g):=\mathcal{E}^{n,\replaced{\beta}{s},\varrho}\big(f,g\big):=\sum_{\varnothing\not=B\subset I}\mathcal{E}_{\scriptscriptstyle{B}}(f,g),\quad f,g\in \mathcal{D}:=C_c^2\big(E\big),
\end{align}
with 
\begin{align*}
\mathcal{E}_{\scriptscriptstyle{B}}(f,g):=\mathcal{E}_{\scriptscriptstyle{B}}^{n,\replaced{\beta}{s},\varrho}\big(f,g\big):=\int_{E_{\scriptscriptstyle{+}}(B)} \big( \nabla^B f\,, \nabla^B g \big) \,d\mu_{\scriptscriptstyle{B}}^{\varrho,n,\replaced{\beta}{s}},\quad\varnothing\not=B\subset I,
\end{align*}
where $\mu_{\scriptscriptstyle{B}}^{\varrho,n,\replaced{\beta}{s}}:=\varrho\,\lambda_{\scriptscriptstyle{B}}^{n,\replaced{\beta}{s}}$ (see (\ref{defmeasure}))\added{, $(\cdot,\cdot)$ denotes the euclidean inner product} and $\nabla^B f := \sum_{i \in B} \partial_i f \ e_i$ for $f \in \mathcal{D}$.

\begin{remark}\label{propdense}
Suppose that Condition \ref{conddensity} is satisfied. Then $\big(\mathcal{E},\mathcal{D}\big)$ is a symmetric, positive definite bilinear form which is densely defined on $L^2\big(E;\mu_{n,\replaced{\beta}{s},\varrho}\big)$.
\end{remark}
\deleted{proof}

To prove closability of the underlying bilinear form, we have to put an additional restriction on the density $\varrho$. For $\varnothing\not= B\subset I$ we define
\begin{align*}
R_\varrho\big(\overline{E_{\scriptscriptstyle{+}}(B)}\big):=\left\{x\in \overline{E_{\scriptscriptstyle{+}}(B)}\,\left|\,\int_{\overline{B_{\varepsilon}(x)}}\varrho^{-1}\,d\lambda_{\scriptscriptstyle{B}}^{n,\replaced{\beta}{s}}<\infty\quad\text{for some}\quad\varepsilon>0\right.\right\},
\end{align*}
where $B_{\varepsilon}(x):=\big\{y\in \overline{E_{\scriptscriptstyle{+}}(B)}\,\big|\,\vert x-y\vert<\varepsilon\big\}$ \deleted{and $|\cdot|_{\scriptscriptstyle{\text{euc}}}$ denotes the norm induced by the euclidean scalar product on $\mathbb{R}^n$} and for $\varnothing\not=B\subset I$, $\overline{E_{\scriptscriptstyle{+}}(B)}$ is the closure of ${E_{\scriptscriptstyle{+}}(B)}$ with respect to $\vert\cdot\vert$. 

\begin{condition}\label{condHamza}
For $\varnothing\not=B\subset I$ we have that $\varrho=0$~$\lambda_{\scriptscriptstyle{B}}^{n,\replaced{\beta}{s}}$-a.e.~on $\overline{E_{\scriptscriptstyle{+}}(B)}\setminus R_\varrho\big(\overline{E_{\scriptscriptstyle{+}}(B)}\big)$.
\end{condition}

\begin{lemma}\label{lemMaRo}
Let Condition \ref{condHamza} be satisfied. For $\varnothing\not=B\subset I$ let $\varphi\in C^\infty_c\Big(R_\varrho\big(\overline{E_{\scriptscriptstyle{+}}(B)}\big)\Big)$ and $f\in L^2\big(E,\mu_{\varrho,n,\replaced{\beta}{s}}\big)$.
\begin{enumerate}
\item[(i)]
There exists $C_1(\varphi,B)\in(0,\infty)$ such that
\begin{align*}
\Bigg|\int_{R_\varrho(\overline{E_{\scriptscriptstyle{+}}(B)})}f\varphi\,d\lambda_{\scriptscriptstyle{B}}^{n,\replaced{\beta}{s}}\Bigg|\le C_1(\varphi,B)\cdot\Vert f\Vert_{L^2(\overline{E_{\scriptscriptstyle{+}}(B)};\mu_{\scriptscriptstyle{B}}^{\varrho,n,\replaced{\beta}{s}})}.
\end{align*}
Here $L^2\big(\overline{E_{\scriptscriptstyle{+}}(B)};\mu_{\scriptscriptstyle{B}}^{\varrho,n,\replaced{\beta}{s}}\big)$ denotes the spaces of square integrable functions on $\overline{E_{\scriptscriptstyle{+}}(B)}$ with respect to $\mu_{\scriptscriptstyle{B}}^{\varrho,n,\replaced{\beta}{s}}$.
\item[(ii)]
There exists $C_2(\varphi,B)\in(0,\infty)$ such that
\begin{align*}
\Bigg|\int_{\partial R_\varrho(\overline{E_{\scriptscriptstyle{+}}(B)})}f\varphi\,d\sigma_{\scriptscriptstyle{B}}^{n,\replaced{\beta}{s}}\Bigg|\le C_2(\varphi,B)\cdot\Vert f\Vert_{L^2(E;\mu_{\varrho,n,\replaced{\beta}{s}})},
\end{align*}
where $\sigma_{\scriptscriptstyle{B}}^{n,\replaced{\beta}{s}}:=\sum_{\hat{B}\subsetneq B}\lambda_{\scriptscriptstyle{\hat{B}}}^{n,\replaced{\beta}{s}}$.
\end{enumerate}
\end{lemma}

\begin{proof}
\begin{enumerate}
\item[(i)]
See e.g.~\cite[Chap.~2,~Lemma 2.2]{MR92}.
\item[(ii)]
Let $\varnothing\not=B\subset I$, $\varphi\in C^\infty_c\Big(R_\varrho\big(\overline{E_{\scriptscriptstyle{+}}(B)}\big)\Big)$ and $f\in L^2\big(E,\mu_{\varrho,n,\replaced{\beta}{s}}\big)$. By a multiple application of part (i) we obtain
\begin{multline*}
\Bigg|\int_{\partial R_\varrho(\overline{E_{\scriptscriptstyle{+}}(B)})}f\varphi\,d\sigma_{\scriptscriptstyle{B}}^{n,\replaced{\beta}{s}}\Bigg|\le
\sum_{\hat{B}\subsetneq B}\Bigg|\int_{\partial R_\varrho(\overline{E_{\scriptscriptstyle{+}}(B)})}f\varphi\,d\lambda_{\scriptscriptstyle{\hat{B}}}^{n,\replaced{\beta}{s}}\Bigg|\\
\le\sum_{\hat{B}\subsetneq B}C_1(\varphi,\hat{B})\cdot\Vert f\Vert_{L^2(\overline{E_{\scriptscriptstyle{+}}(\hat{B})};\mu_{\scriptscriptstyle{\hat{B}}}^{\varrho,n,\replaced{\beta}{s}})}
\le C_2(\varphi,B)\cdot\Vert f\Vert_{L^2(E;\mu_{\varrho,n,\replaced{\beta}{s}})}.
\end{multline*}
\end{enumerate}
\end{proof}

\begin{proposition}\label{propclos}
Suppose that Conditions \ref{conddensity} and \ref{condHamza} are satisfied. Then $\big(\mathcal{E},\mathcal{D}\big)$ is closable on $L^2\big(E;\mu_{\varrho,n,\replaced{\beta}{s}}\big)$. Its closure we denote by $\big(\mathcal{E},D(\mathcal{E})\big)$.
\end{proposition}

\begin{proof}
Let $(f_k)_{k\in\mathbb{N}}$ be a Cauchy sequence in $\mathcal{D}$ with respect to $\mathcal{E}$, i.e.,~$\mathcal{E}(f_k-f_l,f_k-f_l)\to 0$ as $k,l\to\infty$. Furthermore, we suppose that $f_k\to 0$ in $L^2\big(E;\mu_{\varrho,n,\replaced{\beta}{s}}\big)$ as $k\to\infty$, i.e.,
$\big(f_k,f_k\big)_{L^2(E;\mu_{\varrho,n,\replaced{\beta}{s}})}\to 0$ as $k\to\infty$. We have to check whether $\mathcal{E}(f_k,f_k)\to 0$ as $k\to\infty$. 
Let $\varnothing\not=B\subset I$. We know that for fixed $i\in B$, $\big(\partial_i f_k\big)_{k\in\mathbb{N}}$ converges to some $h_i$ in $L^2\big(\overline{E_{\scriptscriptstyle{+}}(B)};\mu_{\scriptscriptstyle{B}}^{\varrho,n,\replaced{\beta}{s}}\big)$, since $\big(\partial_i f_k\big)_{k\in\mathbb{N}}$ is a Cauchy sequence in $L^2\big(\overline{E_{\scriptscriptstyle{+}}(B)};\mu_{\scriptscriptstyle{B}}^{\varrho,n,\replaced{\beta}{s}}\big)$ and $\big(L^2\big(\overline{E_{\scriptscriptstyle{+}}(B)};\mu_{\scriptscriptstyle{B}}^{\varrho,n,\replaced{\beta}{s}}\big),\Vert\cdot\Vert_{L^2(\overline{E_{\scriptscriptstyle{+}}(B)};\mu_{\scriptscriptstyle{B}}^{\varrho,n,\replaced{\beta}{s}})}\big)$ is complete. Let $\varphi\in C_c^\infty\big(R_\varrho\big(\overline{E_{\scriptscriptstyle{+}}(B)}\big)\big)$. Using Lemma \ref{lemMaRo}(i) we obtain that    
\begin{multline*}
\Bigg|\int_{R_\varrho(\overline{E_{\scriptscriptstyle{+}}(B)})}h_i\,\varphi\,d\lambda_{\scriptscriptstyle{B}}^{n,\replaced{\beta}{s}}-\int_{R_\varrho(\overline{E_{\scriptscriptstyle{+}}(B)})}\partial_if_k\,\varphi\,d\lambda_{\scriptscriptstyle{B}}^{n,\replaced{\beta}{s}}\Bigg|\\
=\Bigg|\int_{R_\varrho(\overline{E_{\scriptscriptstyle{+}}(B)})}\Big(h_i-\partial_i f_k\Big)\,\varphi\,d\lambda_{\scriptscriptstyle{B}}^{n,\replaced{\beta}{s}}\Bigg|\le C_1(\varphi,B)\cdot\Vert h_i-\partial_i f_k\Vert_{L^2(\overline{E_{\scriptscriptstyle{+}}(B)};\mu_{\scriptscriptstyle{B}}^{\varrho,n,\replaced{\beta}{s}})}\to 0\quad\text{as}\quad k\to\infty.
\end{multline*}
This, together with an integration by parts, triangle inequality, Lemma \ref{lemMaRo}(ii) and the fact that $\big(f_k,f_k\big)_{L^2(E;\mu_{\varrho,n,\replaced{\beta}{s}})}\to 0$ as $k\to\infty$ implies:
\begin{multline*}
\left|\int_{R_\varrho(\overline{E_+(B)})}h_i\,\varphi\,d\lambda_{\scriptscriptstyle{B}}^{n,\replaced{\beta}{s}}\right|=\lim_{k\to\infty} \left|\int_{R_\varrho(\overline{E_+(B)})}\partial_if_k\,\varphi\,d\lambda_{\scriptscriptstyle{B}}^{n,\replaced{\beta}{s}}\right|\\
=\lim_{k\to\infty} \left|\int_{\partial(R_\varrho(\overline{E_+(B)}))}f_k\,\varphi\,d\sigma_{\scriptscriptstyle{B}}^{n,\replaced{\beta}{s}}-\int_{R_\varrho(\overline{E_+(B)})} f_k\,\partial_i\varphi\,d\lambda_{\scriptscriptstyle{B}}^{n,\replaced{\beta}{s}}\right|\\
\le\lim_{k\to\infty} \left|\int_{\partial(R_\varrho(\overline{E_+(B)}))}f_k\,\varphi\,d\sigma_{\scriptscriptstyle{B}}^{n,\replaced{\beta}{s}}\right|+\lim_{k\to\infty}\left|\int_{R_\varrho(\overline{E_+(B)})} f_k\,\partial_i\varphi\,d\lambda_{\scriptscriptstyle{B}}^{n,\replaced{\beta}{s}}\right|
=0\quad\text{as}\quad k\to\infty.
\end{multline*}
Thus $h_i=0$ in $L^2\big(R_\varrho(\overline{E_{\scriptscriptstyle{+}}(B)});\lambda_{\scriptscriptstyle{B}}^{n,\replaced{\beta}{s}}\big)$ and therefore $h_i=0$ in $L^2\big(\overline{E_{\scriptscriptstyle{+}}(B)};\varrho\,\lambda_{\scriptscriptstyle{B}}^{n,\replaced{\beta}{s}}\big)$ by Condition \ref{condHamza}. For all $\varnothing\not=B\subset I$ this yields $h_i=0$ in $L^2\big(\overline{E_{\scriptscriptstyle{+}}(B)};\mu_{\scriptscriptstyle{B}}^{\varrho,n,\replaced{\beta}{s}}\big)$ for all $i\in B$. Moreover,
\begin{multline*}
\mathcal{E}(f_k,f_k)=\sum_{\varnothing\not=B\subset I}\int_{E_{\scriptscriptstyle{+}}(B)}\big|\nabla^B f_k \big|^2\,d\mu_{\scriptscriptstyle{B}}^{\varrho,n,\replaced{\beta}{s}}\\
=\sum_{\varnothing\not=B\subset I}\sum_{i\in B}\big\Vert\partial_i f_k-h_i\big\Vert_{L^2(\overline{E_{\scriptscriptstyle{+}}(B)};\mu_{\scriptscriptstyle{B}}^{\varrho,n,\replaced{\beta}{s}})}^2\to 0\quad\text{as}\quad k\to\infty
\end{multline*}
and closability is shown. 
\end{proof}

\begin{remark}
Since $\big(\mathcal{E},\mathcal{D}\big)$ is closable on $L^2\big(E;\mu_{\varrho,n,\replaced{\beta}{s}}\big)$ by Proposition \ref{propclos} we have that $D(\mathcal{E})$ is complete with respect to the norm $\Vert\cdot\Vert_{\mathcal{E}_1}:={\mathcal{E}(\cdot,\cdot)}^{\frac{1}{2}}+{\big(\cdot,\cdot\big)^{\frac{1}{2}}_{L^2(E;\mu_{\varrho,n,\replaced{\beta}{s}})}}$.
\end{remark}

\begin{proposition}\label{propDirichlet}
Suppose that Conditions \ref{conddensity} and \ref{condHamza} are satisfied. Then $\big(\mathcal{E},D(\mathcal{E})\big)$ is a symmetric, regular\added{, strongly local and recurrent, hence in particular conservative,} Dirichlet form.
\end{proposition}
\deleted{next proposition with proof}
\begin{proof}
The Markov property is clear, see e.g. \replaced{\cite[Theo.~1.4.1]{FOT94}}{MR92}. Regularity can be shown as follows.
The extended Stone--Weierstra\ss~theorem, see e.g.~\cite[Chap.~7, Sect.~38]{Sim63}, yields that $C_c^\infty\big(E\big)$ is dense in $C_c^0\big(E\big)$ with respect to $\Vert\cdot\Vert_{\sup}$. Furthermore, $\mathcal{D}$ is dense in $D\big(\mathcal{E}\big)$ with respect to $\Vert\cdot\Vert_{\mathcal{E}_1}$. Since $C_c^\infty(E)\subset\mathcal{D}\subset D(\mathcal{E})\cap C^0_c(E)$, we obtain that $\big(\mathcal{E},D(\mathcal{E})\big)$ is regular.
Using \cite[Theo.~3.1.1]{FOT94} and \cite[Exercise~3.1.1]{FOT94} it is sufficient to show the strong local property for elements in $\mathcal{D}$. Therefore, let $f,g\in\mathcal{D}$ with $\supp(f)$, $\supp(g)$ compact and let $g$ be constant on some open (in the trace topology of $E$) neighborhood $U$ of $\supp(f)$. Then
\begin{multline*}
\mathcal{E}\big(f,g\big)=\sum_{\varnothing\not=B\subset I}\int_{E_{\scriptscriptstyle{+}}(B)} \big(\nabla^B f\,,\nabla^B g\big)\,d\mu_{\scriptscriptstyle{B}}^{\varrho,n,\replaced{\beta}{s}}\\
=\sum_{\varnothing\not=B\subset I}\int_{E_{\scriptscriptstyle{+}}(B)\cap\supp(f)}\big(\nabla^B f\,,\underbrace{\nabla^B g}_{=0}\big)\,d\mu_{\scriptscriptstyle{B}}^{\varrho,n,\replaced{\beta}{s}}+\sum_{\varnothing\not=B\subset I} \int_{E_{\scriptscriptstyle{+}}(B)\setminus\supp(f)}\big(\underbrace{\nabla^B f}_{=0}\,,\nabla^B g\big)\,d\mu_{\scriptscriptstyle{B}}^{\varrho,n,\replaced{\beta}{s}}=0.
\end{multline*} 
Hence $\big(\mathcal{E},D(\mathcal{E})\big)$ is strongly local.

\added{In order to deduce recurrence of $(\mathcal{E},D(\mathcal{E}))$, it is enough to show that there exists a sequence $(f_k)_{k \in \mathbb{N}} \subset D(\mathcal{E})$ such that $\lim_{k \rightarrow \infty} f_k =1$ $\mu_{\varrho,n,\replaced{\beta}{s}}$-a.e. and $\lim_{k \rightarrow \infty} \mathcal{E}(f_k,f_k)=0$ by \cite[Theorem 1.6.3]{FOT94}. This we do next.}
\deleted{Next we prove \replaced{conservativeness}{conservativity}.} $\mathbbm{1}_{E}\in L^2\big(E;\mu_{\varrho,n,\replaced{\beta}{s}}\big)$ by Condition \ref{conddensity}. We show that $\mathbbm{1}_{E}\in D(\mathcal{E})$. Set $\Lambda:=[-1,\infty)^n$ and $K_k:=[0,k]^n$, $k\in\mathbb{N}$. Then there exist cutoff functions $f_k\in C_c^\infty(\Lambda)$, $k\in\mathbb{N}$, such that $0\le f_k\le f_{k+1}\le 1$, $f_k=1$ on $K_k$, $\supp(f_k)\subset B_1(K_k)$ and $\big|\partial_i f_k\big|\le C_3<\infty$. $C_3$ independent of $k\in\mathbb{N}$. Here $B_1(K_k)$, $k\in\mathbb{N}$, denotes the $1$-neighborhood of the set $K_k$. Hence
\begin{multline}\label{conv1}
\Vert\mathbbm{1}_{E}-f_k\Vert_{L^2(E;\mu_{\varrho,n,\replaced{\beta}{s}})}^2=\int_{E}\big(\mathbbm{1}_{E}-f_k\big)^2\,d\mu_{\varrho,n,\replaced{\beta}{s}}\\
=\int_{E\setminus K_k}\big(\mathbbm{1}_{E}-f_k\big)^2\,d\mu_{\varrho,n,\replaced{\beta}{s}}\le \mu_{\varrho,n,\replaced{\beta}{s}}\big(E\setminus K_k\big)\to 0\quad\text{as}\quad k\to\infty.
\end{multline}
Furthermore,
\begin{multline}\label{conv2}
\mathcal{E}(f_k,f_k)=\sum_{\varnothing\not=B\subset I}\int_{E_{\scriptscriptstyle{+}}(B)}\big|\nabla^B f_k\big|^2\,d\mu_{\scriptscriptstyle{B}}^{\varrho,n,\replaced{\beta}{s}}
\le n\,C_3^2\,\mu_{\varrho,n,\replaced{\beta}{s}}\big(E\setminus K_k\big)\to 0\quad\text{as}\quad k\to\infty.
\end{multline} 
Using (\ref{conv1}) and (\ref{conv2}) we easily obtain by applying the Cauchy-Schwarz inequality that $(f_k)_{k\in\mathbb{N}}$ is $\mathcal{E}_1$-Cauchy. Hence $\mathbbm{1}_{E}\in D(\mathcal{E})$ with
\begin{align*}
\mathcal{E}(\mathbbm{1}_E,\mathbbm{1}_E)=\lim_{k\to\infty}\mathcal{E}(f_k,f_k)=0.
\end{align*}
\added{Therefore, $(\mathcal{E},D(\mathcal{E}))$ is recurrent and, hence in particular, conservative}.
\end{proof}




Finally, we end up with the following result.

\begin{theorem}\label{theosumdiri}
For fixed $n\in\mathbb{N}$, $\replaced{\beta}{s}\in(0,\infty)$ and density function $\varrho$ we have that under Conditions \ref{conddensity} and \ref{condHamza} 
\begin{align*}
\mathcal{E}(f,g)=\sum_{\varnothing\not=B\subset I}\mathcal{E}_{\scriptscriptstyle{B}}(f,g),\quad f,g\in \mathcal{D}=C_c^2\big(E\big),
\end{align*}
with 
\begin{align*}
\mathcal{E}_{\scriptscriptstyle{B}}(f,g)=\int_{E_{\scriptscriptstyle{+}}(B)}\big(\nabla^B f\,,\nabla^B g\big)\,d\mu_{\scriptscriptstyle{B}}^{\varrho,n,\replaced{\beta}{s}},\quad\varnothing\not=B\subset I,
\end{align*}
and $\mu_{\scriptscriptstyle{B}}^{\varrho,n,\replaced{\beta}{s}}=\varrho\,\lambda_{\scriptscriptstyle{B}}^{n,\replaced{\beta}{s}}$,
is a densely defined, positive definite, symmetric bilinear form, which is closable on $L^2\big(E;\mu_{\varrho,n,\replaced{\beta}{s}}\big)$. Its closure $\big(\mathcal{E},D(\mathcal{E})\big)$ is a \added{recurrent, hence in particular conservative}, strongly local, regular, symmetric Dirichlet form on $L^2\big(E;\mu_{\varrho,n,\replaced{\beta}{s}}\big)$.
\end{theorem}
\begin{proof}
See Remark \ref{propdense} and Propositions \ref{propclos} and \ref{propDirichlet}.
\end{proof}

\subsection{Generators}
By Friedrichs representation theorem we have the existence of the self-adjoint generator\\
$\big(H,D(H)\big)$ corresponding to $\big(\mathcal{E},D(\mathcal{E})\big)$.
\begin{proposition}\label{propgen}
Suppose that Conditions \ref{conddensity} and \ref{condHamza} are satisfied. There exists a unique, positive, self-adjoint, linear operator $\big(H,D(H)\big)$ on $L^2\big(E;\mu_{\varrho,n,\replaced{\beta}{s}}\big)$ such that
\begin{align*}
D(H)\subset D(\mathcal{E})\quad\text{and}\quad\mathcal{E}\big(f,g\big)=\Big(H f,g\Big)_{L^2(E;\mu_{\varrho,n,\replaced{\beta}{s}})}
\quad\text{for all }f\in D(H),~g\in D(\mathcal{E}).
\end{align*}
\end{proposition}
\begin{proof}
Using Proposition \ref{propclos} this is a direct application of \cite[Coro.~1.3.1]{FOT94}.
\end{proof}

We need additional assumptions on the density function $\varrho$ in order to derive an explicit formula for the generator $H$ on a subset of its domain $D(H)$, dense in $L^2\big(E;\mu_{\varrho,n,\replaced{\beta}{s}}\big)$. 

\begin{condition}\label{condweakdiff}
\added{$\varrho$ is a $m_{n,\replaced{\beta}{s}}$-a.e. positive function on ${E}$ such that}
\begin{enumerate}
\item[(i)]
$\sqrt{\varrho\added{\big|_{E_+(B)}}}\in H^{1,2}(\replaced{E_+(B)}{\mathring{E}})$ \added{for all $\varnothing \not= B\subset I$}, where $H^{1,2}(\replaced{E_+(B)}{\mathring{E}})$ denotes the Sobolev space of weakly differentiable functions on $\replaced{E_+(B)}{\mathring{E}}$, square integrable together with their derivative.
\item[(ii)]
$\varrho\in C^1(E)$, where $C^1(E)$ denotes the space of continuously differentiable functions on $E$.
\end{enumerate}
\end{condition}

\begin{remark}\label{remcondequi}
$ $
\begin{enumerate}
\item[(i)]
Note that the additional assumptions collected in Condition \ref{condweakdiff} are not necessary for the existence of the generator $\big(H,D(H)\big)$.
\item[(ii)]
If $\varrho$ fulfills Condition \ref{conddensity} then Condition \ref{condweakdiff}(i) is equivalent to $\big(\partial_i\ln(\varrho)\big)_{i=1}^n$\\
$\in L^2(E;\mu^{\varrho,n,\replaced{\beta}{s}})$.
\item[(iii)]
Condition \ref{condweakdiff}(ii) \replaced{implies that}{is equivalent to} $\big(\partial_i\ln(\varrho)\big)_{i=1}^n$ \replaced{is}{being} continuous \added{on the set $\{ \varrho >0\}$}.
\item[(iv)] \added{If $\varrho$ fulfills Condition \ref{condweakdiff} (ii), $\varrho$ is in particular continuous on $E$ and therefore Condition \ref{condHamza} is implied. Moreover, Condition \ref{condweakdiff} (i) implies Condition \ref{conddensity}.}
\end{enumerate}
\end{remark}

For $f\in \mathcal{D}=C_c^2(E)$ and $B \subset I$ we define
\begin{multline*}
L^Bf:={L}^{n,\varrho,B} f:=\sum_{i \in B} \big( \partial^2_i f+ \partial_i f\,\partial_i (\ln\varrho) \big)+ \sum_{i \in I \backslash B} \frac{1}{\replaced{\beta}{s}} \partial_i f\\
 = \Delta^B f + \big( \nabla^B  f, \nabla^B \ln \varrho \big) + \frac{1}{\replaced{\beta}{s}} (\nabla^{I \backslash B} f,e),
\end{multline*}
and
\begin{align*} Lf := \sum_{B \subset I} \mathbbm{1}_{E_+(B)} L^B f,
\end{align*}
where $\Delta^B f:=\sum_{i \in B} \partial_i^2 f$ for $f \in \mathcal{D}$, $B \subset I$ and $e$ is a vector of length $n$ containing only ones.

\begin{proposition}\label{propibp}
Suppose that Condition\deleted{s} \deleted{\ref{conddensity}, \ref{condHamza} and} \ref{condweakdiff} \replaced{is}{are} satisfied. For functions $f, g\in \mathcal{D}$
we have the representation $\mathcal{E}\big(f,g\big)=\Big(-Lf,g\Big)_{\scriptscriptstyle{L^2(E;\mu_{\varrho,n,\replaced{\beta}{s}})}}$.
\end{proposition}

\begin{remark}\label{sticky}
Let $L_1^B:=\Delta^B  + \big( \nabla^B  , \nabla^B \ln \varrho \big)$ and $L_2^{\replaced{I\setminus B}{B}} :=(\nabla^{\replaced{I\setminus B}{B}},e)$. Using this notation we can express $L$ in the form
\begin{align*} Lf&= \sum_{B \subset I} \mathbbm{1}_{E_+(B)}  (L_1^B f + \frac{1}{\replaced{\beta}{s}} L_2^{I \backslash B} f) \\
&= \added{\mathbbm{1}_{E_+(I)}\cdot}L_1^I f + \sum_{B \subsetneq I} \mathbbm{1}_{E_+(B)} (-L_1^{I \backslash B}f + \frac{1}{\replaced{\beta}{s}} L_2^{I \backslash B}f),\quad f\in\mathcal{D}.
\end{align*}
The interpretation of $L$ is that on $E_+(B)$ the operator $L_1^B$ describes the dynamics of the coordinates $i \in B$ by means of a diffusive and a drift term whereas the operator $\frac{1}{\replaced{\beta}{s}} L_2^{I \backslash B}$ forces the remaining coordinates $i \in I \backslash B$ with constant drift $\frac{1}{\replaced{\beta}{s}}$ back to positive height. The operator $-L_1^B + \frac{1}{\replaced{\beta}{s}} L_2^B$ for $B \neq \varnothing$ is called a \textit{Wentzell type boundary operator}. The associated Cauchy problem can be formulated in the form
\begin{align}
\label{PDE}
\left\{
\begin{array}{l l}
& \frac{\partial}{\partial t} u_t(x) = \Delta u_t(x) + \big(\nabla u_t(x), \nabla (\ln \varrho)(x)\big), \quad \quad t>0, \ x \in E,  \\
& \partial_i^2 u_t(x) + \partial_i u_t(x) \partial_i (\ln \varrho)(x)  - \frac{1}{\replaced{\beta}{s}} \partial_i u_t(x)=0, \quad t >0, i \in I,  \ x \in E \cap \{x_i=0\}, \\
& u_0(x)=f(x)
\end{array}
\right.
\end{align}
The second line of (\ref{PDE}) is called \textit{Wentzell boundary condition (for the $i$-th coordinate)}.
\end{remark}

\begin{proof}[Proof of Proposition \ref{propibp}]
Let $f\in \replaced{\mathcal{D}}{C_c^2(E)}$ and $g \in \deleted{\mathcal{D}=}C_c^1(E)$. In order to show this representation we carry out an integration by parts. We start with $B=I$, i.e., $\# B=n$:
\begin{multline*}
\mathcal{E}_{\scriptscriptstyle{I}}(f,g)=\sum_{i\in I}\int_{\mathring{E}}\partial_if\,\partial_i g\,\varrho\,d\lambda^{\scriptscriptstyle{(n)}}_{\scriptscriptstyle{I}}
=\sum_{i\in I}\int_{\mathring{E}}\partial_if\varrho\,\partial_i g\,d\lambda_{\scriptscriptstyle{I}}^{\scriptscriptstyle{(n)}}\\
=\sum_{i\in I}\int_{\mathring{E}}\Big(-\partial^2_if\varrho-\partial_if\partial_i\varrho\Big)\,
g(x)\,d\lambda^{\scriptscriptstyle{(n)}}_{\scriptscriptstyle{I}}-\sum_{\stackunder{\#B=n-1}{B\subset I}}\sum_{i\in I\setminus B}\int_{E_+(B)}\partial_i fg\varrho\,d\lambda^{\scriptscriptstyle{(n)}}_{\scriptscriptstyle{B}}\\
=\sum_{i\in I}\int_{\scriptscriptstyle{\mathring{E}}}\Big(-\partial^2_i f-\partial_i f\partial_i \ln(\varrho)\Big)\,
g(x)\,\varrho\,d\lambda_{\scriptscriptstyle{I}}^{\scriptscriptstyle{(n)}}-\sum_{\stackunder{\#B=n-1}{B\subset I}}\sum_{i\in I\setminus B}\int_{E_+(B)}\partial_i f\,g\varrho\,d\lambda_{\scriptscriptstyle{B}}^{\scriptscriptstyle{(n)}}.
\end{multline*}
Next we consider all $B\subset I$ such that $\# B=n-1$, i.e.,
\begin{multline*}
\mathcal{E}_{\scriptscriptstyle{B}}(f,g)=\sum_{i\in B}\int_{E_{+}(B)}\partial_if\,\partial_i g\,\replaced{\beta}{s}\,\varrho\,\prod_{i\in B}dx_{+}^i\prod_{j\in I\setminus B}d\delta^j_0\\
=\sum_{i\in B}\int_{E_+(B)}\Big(-\partial^2_i f-\partial_i f\partial_i \ln(\varrho)\Big)\,
g\,\varrho\,\replaced{\beta}{s}\,d\lambda_{\scriptscriptstyle{B}}^{\scriptscriptstyle{(n)}}\\
-\sum_{\stackunder{\#\tilde{B}=n-2}{\tilde{B}\subset B}}\sum_{i\in B\setminus \tilde{B}}\int_{E_+(\tilde{B})}\partial_i f\,g\varrho\,\replaced{\beta}{s}\,d\lambda_{\scriptscriptstyle{\tilde{B}}}^{\scriptscriptstyle{(n)}}.
\end{multline*}
Proceeding inductively we end up with all $B\subset I$ fulfilling $\# B=1$, i.e., we consider
\begin{multline*}
\mathcal{E}_{\scriptscriptstyle{B}}(f,g)=\sum_{i\in B}\int_{E_{+}(B)}\partial_if\,\partial_i g\,\varrho\,\replaced{\beta}{s}^{n-1}\,\prod_{i\in B}dx_{+}^i\prod_{j\in I\setminus B}d\delta^j_0\\
=\sum_{i\in B}\int_{E_+(B)}\Big(-\partial^2_i f-\partial_i f\partial_i \ln(\varrho)\Big)\,
g\,\varrho\,\replaced{\beta}{s}^{n-1}\,d\lambda_{\scriptscriptstyle{B}}^{\scriptscriptstyle{(n)}}\\
-\replaced{\beta}{s}^{n-1}\,\sum_{i\in B}\partial_i f(0)\,g(0)\varrho(0).
\end{multline*}
Combining all this yields
\begin{multline}\label{ibp}
\mathcal{E}\big(f,g\big)=\sum_{\varnothing\not=B\subset I}\mathcal{E}_{\scriptscriptstyle{B}}\big(f,g\big)\\
=\sum_{i\in I}\int_{\mathring{E}}\Big(-\partial^2_i f-\partial_i f\partial_i \ln(\varrho)\Big)\,
g\,\varrho\,d\lambda_{\scriptscriptstyle{I}}^{\scriptscriptstyle{(n)}}\\
+\sum_{\stackunder{\#B=n-1}{B\subset I}}\int_{E_+(B)}\left(\sum_{i\in B}\Big(-\partial^2_i f-\partial_i f\partial_i \ln(\varrho)\Big)-\frac{1}{\replaced{\beta}{s}}\sum_{i\in I\setminus B}\partial_i f\right)\,
\,g\,\replaced{\beta}{s}\,\varrho\,d\lambda_{\scriptscriptstyle{B}}^{\scriptscriptstyle{(n)}}\\
+\sum_{\stackunder{\#B=n-2}{B\subset I}}\int_{E_+(B)}\left(\sum_{i\in B}\Big(-\partial^2_i f-\partial_i f\partial_i \ln(\varrho)\Big)-\frac{1}{\replaced{\beta}{s}}\sum_{i\in I\setminus B}\partial_i f\right)\,
\,g\,\replaced{\beta}{s}^2\,\varrho\,d\lambda_{\scriptscriptstyle{B}}^{\scriptscriptstyle{(n)}}\\
\begin{array}{c}
+\\
\vdots\\
+
\end{array}\\
\sum_{\stackunder{\#B=1}{B\subset I}}\int_{E_+(B)}\left(\sum_{i\in B}\Big(-\partial^2_i f-\partial_i f\partial_i \ln(\varrho)\Big)-\frac{1}{\replaced{\beta}{s}}\sum_{i\in I\setminus B}\partial_i f\right)\,
\,g\,\replaced{\beta}{s}^{n-1}\,\varrho(x)\,d\lambda_{\scriptscriptstyle{B}}^{\scriptscriptstyle{(n)}}\\
-\sum_{i\in I}\frac{1}{\replaced{\beta}{s}}\,\partial_i f(0)\,g(0)\,\replaced{\beta}{s}^n\,\varrho(0).
\end{multline}
Now using the definition of $L$, we obtain the desired result.
\end{proof}

\section{The associated Markov process}\label{sectprocess}
Since $\big(\mathcal{E},D(\mathcal{E})\big)$ is a regular, symmetric Dirichlet form on $L^2\big(E;\mu_{\varrho,n,\replaced{\beta}{s}}\big)$ which is \added{recurrent, hence in particular} conservative, and possesses the strong local property, we obtain the following theorem, where $\big(T_t\big)_{t>0}$ denotes the $C^0$-semigroup corresponding to $\big(\mathcal{E},D(\mathcal{E})\big)$, see e.g.~\deleted{MR}
\cite[Chap.~4 and Chap.~7]{FOT94}.
\begin{theorem}\label{theoprocess}
Suppose that Conditions \ref{conddensity} and \ref{condHamza} are satisfied. Then there exists a conservative diffusion process (i.e.~a strong Markov process with continuous sample paths and infinite life time)
\begin{align*}
\mathbf{M}:=\mathbf{M}^{\varrho,n,\replaced{\beta}{s}}:=\big(\mathbf{\Omega},\mathbf{F},(\mathbf{F}_t)_{t\ge 0},(\mathbf{X}_t)_{t\ge 0},(\mathbf{\Theta}_t)_{t\ge 0},(\mathbf{P}^{\varrho,n,\replaced{\beta}{s}}_x)_{x\in E}\big)
\end{align*}
with state space $E$ which is \deleted{properly} associated with $\big(\mathcal{E},D(\mathcal{E})\big)$, i.e., for all ($\mu_{\varrho,n,\replaced{\beta}{s}}$-versions of) $f\in L^2\big(E;\mu_{\varrho,n,\replaced{\beta}{s}}\big)$ and all $t>0$ the function
\begin{align*}
E\ni x\mapsto p_tf(x):=\mathbb{E}^{\varrho,n,\replaced{\beta}{s}}_x\Big(f\big(\mathbf{X}_t\big)\Big):=\int_{\mathbf{\Omega}}f\big(\mathbf{X}_t\big)\,d\mathbf{P}^{\varrho,n,\replaced{\beta}{s}}_x\in[0,\infty)
\end{align*}
is a \deleted{$\mathcal{E}$-}quasi continuous version of $T_tf$. $\mathbf{M}$ is up to $\mu_{\varrho,n,\replaced{\beta}{s}}$-equivalence unique. In particular, $\mathbf{M}$ is $\mu_{\varrho,n,\replaced{\beta}{s}}$-symmetric, i.e., 
\begin{align*}
\int_{E}p_tf\,g\,d\mu_{\varrho,n,\replaced{\beta}{s}}=\int_{E}f\,p_tg\,d\mu_{\varrho,n,\replaced{\beta}{s}}\quad\text{for all}\quad f,g\replaced{:E\to[0,\infty)\text{ measurable}}{\in L^2\big(E;\mu_{\varrho,n,\replaced{\beta}{s}}\big)}\text{ and all }t>0,
\end{align*}
and has $\mu_{\varrho,n,\replaced{\beta}{s}}$ as \added{reversible,} invariant measure, i.e., 
\begin{align*}
\int_{E}p_tf\,d\mu_{\varrho,n,\replaced{\beta}{s}}=\int_{E}f\,d\mu_{\varrho,n,\replaced{\beta}{s}}\quad\text{for all}\quad f\replaced{:E\to[0,\infty)\text{ measurable}} {\in L^2\big(E;\mu_{\varrho,n,\replaced{\beta}{s}}\big)}\text{ and all }t>0.
\end{align*}
\end{theorem}
In the above theorem $\mathbf{M}$ is canonical, i.e., $\mathbf{\Omega}=C^0\big([0,\infty),E\big)$, the space of continuous functions on $[0,\infty)$ into $E$, $\mathbf{X}_t(\omega)=\omega(t)$, $\omega\in\mathbf{\Omega}$. The filtration $(\mathbf{F}_t)_{t\ge 0}$ is the natural minimum completed admissible filtration obtained from the $\sigma$-algebras $\mathbf{F}^0_t:=\sigma\Big\{\textbf{X}_{\replaced{s}{\tau}}\,\Big|\,0\le \replaced{s}{\tau}\le t\Big\}$, $t\ge 0$, and $\mathbf{F}:=\mathbf{F}_{\scriptscriptstyle{\infty}}:=\bigvee_{t\in[0,\infty)}\mathbf{F}_t$. For each $t\ge 0$ we denote by $\mathbf{\Theta}_t:\mathbf{\Omega}\to\mathbf{\Omega}$ a shift operator such that $\mathbf{X}_{\replaced{s}{\tau}}\circ\mathbf{\Theta}_t=\mathbf{X}_{\replaced{s}{\tau}+t}$ for all $\replaced{s}{\tau}\ge 0$.
\begin{proof}
See e.g.~\deleted{MR}
\cite[Theo.~7.2.2 and Exercise 4.5.1~]{FOT94}.
\end{proof}

\begin{theorem}\label{theomartingale}
The diffusion process $\mathbf{M}$ from Theorem \ref{theoprocess} is up to $\mu_{\varrho,n,\replaced{\beta}{s}}$-equivalence the unique diffusion process having $\mu_{\varrho,n,\replaced{\beta}{s}}$ as symmetrizing measure and solving the martingale problem for $\big(H,D(H)\big)$, i.e., for all $g\in D(H)$
\begin{align*}
\widetilde{g}(\mathbf{X}_t)-\widetilde{g}(\mathbf{X}_0)+\int_0^t \big(Hg\big)(\mathbf{X}_{\replaced{s}{\tau}})\,d\replaced{s}{\tau},\quad t\ge 0,
\end{align*}
is an $\mathbf{F}_t$-martingale under $\mathbf{P}^{\varrho,n,\replaced{\beta}{s}}_x$ for \deleted{$\mathcal{E}$-}quasi all $x\in E$. \deleted{$\mathcal{E}$-}\replaced{Q}{q}uasi all $x\in E$ or \deleted{$\mathcal{E}$-}quasi every $x\in E$ (abbreviated by q.a.~$x\in E$ or q.e.~$x\in E$, respectively) means all $x\in E$ except those contained in a set of \deleted{$\mathcal{E}$-}capacity zero. 
(Here $\widetilde{g}$ denotes a quasi-continuous version of $g$, see \cite[Chap.~IV, Prop.~3.3]{MR92}.\added{ Moreover, note that in our setting the notions of capacity in the sence of \cite{MR92} and \cite{FOT94} coincide.})
\end{theorem}
\begin{proof}
See e.g.~\cite[Theo.~3.4(i)]{AR95}.
\end{proof}

\section{Analysis of the stochastic process by additive functionals}\label{sectanapro}  
Throughout this section we assume that we are given the regular, symmetric Dirichlet form $\big(\mathcal{E},D(\mathcal{E})\big)$ on $L^2\big(E;\mu_{\varrho,n,\replaced{\beta}{s}}\big)$ which is \added{recurrent, hence in particular} conservative\added{,} and possesses the strong local property, see Section \ref{sectfuana}, and the associated diffusion process $\mathbf{M}$ from Section \ref{sectprocess}. 
Let $g\in D(\mathcal{E})$ be essentially bounded. Due to \cite[Sect.~3.2]{FOT94} there exists a unique, finite, positive Radon measure $\nu_{\scriptscriptstyle{\langle g\rangle}}$ on $\big(E,\mathcal{B}_{E}\big)$ satisfying
\begin{align*}
\int_{E}f\,d\nu_{\scriptscriptstyle{\langle g\rangle}}=2\,\mathcal{E}(gf,g)-\mathcal{E}(g^2,f)\quad\text{for all }f\in D(\mathcal{E})\cap C^{0}_c\big(E\big).
\end{align*}
\begin{remark}
For an essentially bounded $g\in D(\mathcal{E})$ the measure $\nu_{\scriptscriptstyle{\langle g\rangle}}$ is called the \emph{energy measure} of $g$.
\end{remark}

\begin{lemma} \label{lemweakdiff} \added{Suppose that Condition \ref{conddensity} is satisfied and that $\varrho$ is additionally continuous on $E$.} Let $\Omega$ be a relatively compact subset of $E \backslash \{ \varrho =0\}$ such that $\overline{\Omega} \subset E \backslash \{ \varrho =0\}$ and $\varnothing\not=B\subset I$ such that $E_+(B) \cap \Omega$ is non-empty. Then the restriction map $i_B:g \mapsto g|_{E_+(B) \cap \Omega}$ maps continuously from $D(\mathcal{E})$ to $H^{1,2}(E_+(B) \cap \Omega)$, i.e., there exists a constant $C_B=C_B(B,\replaced{\beta}{s},n,\varrho,\Omega)$ such that $\Vert g \Vert_{H^{1,2}(E_+(B) \cap \Omega)} \leq C_B \sqrt{\mathcal{E}_1(g,g)}$
for $g \in D(\mathcal{E})$.
\end{lemma}

\begin{proof} Let $\varnothing \neq B \subset I$ such that $E_+(B) \cap \Omega \neq \varnothing$. By assumption $\overline{\Omega}$ is compact and \deleted{properly} contained in $\{\varrho >0 \}$. Therefore, there exist constants $\varrho^+,\varrho^-\in (0,\infty)$ and  such that $\varrho^- \leq \varrho \leq \varrho^+$ on $\overline{\Omega}$. Let $g \in \mathcal{D}$. Note that $g \in \mathcal{D}$ is (weakly) differentiable on $E_+(B)$ with gradient $\nabla^B g:=\sum_{i\in B}\partial_i g\,e_i$. We have
\begin{align*} & \int_{E_+(B) \cap \Omega} g^2 \ d\lambda_B^{(n)} + \int_{E_+(B) \cap \Omega} |\nabla^B g|^2 \ d\lambda_B^{(n)} \\
&\leq \frac{1}{\replaced{\beta}{s}^{n-\#B} \varrho^-} \left( \int_{E_+(B) \cap \Omega} g^2 \ d\mu_B^{\varrho,n,\replaced{\beta}{s}} + \int_{E_+(B) \cap \Omega} |\nabla^B g|^2 \ d\mu_B^{\varrho,n,\replaced{\beta}{s}} \right) \\
&\leq \frac{1}{\replaced{\beta}{s}^{n-\#B} \varrho^-} \left(\int_{E_+(B)} g^2 \ d\mu_B^{\varrho,n,\replaced{\beta}{s}} + \int_{E_+(B)} |\nabla^B g|^2 \ d\mu_B^{\varrho,n,\replaced{\beta}{s}} \right) \\
&\leq \frac{1}{\replaced{\beta}{s}^{n-\#B} \varrho^-}\, \mathcal{E}_1(u,u) < \infty.
\end{align*}
Hence, $i_B: \mathcal{D} \rightarrow H^{1,2}(E_+(B) \cap \Omega)$ is well-defined and continuous. Therefore, $i_B$ admits a continuous extension to $D(\mathcal{E})$. Let $g \in D(\mathcal{E})$ and $(g_k)_{k \in \mathbb{N}}$ be a sequence in $\mathcal{D}$ converging to $g$ with respect to the $\mathcal{E}_1^{\frac{1}{2}}$-norm. Then $i_B(g_k)=g_k|_{E_+(B) \cap \Omega} \rightarrow v=i_B(g)$ in $H^{1,2}(E_+(B) \cap \Omega)$. In particular, the same holds true with respect to the $L^2(E_+(B) \cap \Omega;\lambda_B^{(n)})$-norm. Certainly, the convergence of $(g_k)_{k \in \mathbb{N}}$ to $g$ implies convergence in $L^2(E;\mu_{\varrho,n,\replaced{\beta}{s}})$ which in turn implies that $g_k|_{E_+(B) \cap \Omega} \rightarrow g|_{E_+(B) \cap \Omega}$ in $L^2(E_+(B) \cap \Omega;\lambda_B^{(n)})$ due to the boundedness of $\varrho$ on $\Omega$. Hence, we can conclude that $i_B(g)=g|_{E_+(B) \cap \Omega}$ by the uniqueness of the limit. Thus, the map
\[ i_B: D(\mathcal{E}) \rightarrow H^{1,2}(E_+(B) \cap \Omega), g \mapsto g|_{E_+(B) \cap \Omega} \]
is well-defined and continuous. Set $C_B := \frac{1}{\sqrt{\replaced{\beta}{s}^{n-\#B} \varrho^{-}}}$.
\end{proof}

\begin{proposition}\label{propenergymeasure}
Suppose that Conditions \ref{conddensity} \deleted{and \ref{condHamza} are} \added{is} satisfied \added{and that $\varrho$ is additionally continuous on $E$}. Let $g\in D(\mathcal{E}) \cap C^0_c(\mathcal{E})$. Then $g$ is weakly differentiable on $E_+(B) \setminus \{ \varrho =0\}$ for each $\varnothing \neq B \subset I$ with gradient $\nabla^B g \in L^2(E_+(B);\mu_B^{\varrho,n,\replaced{\beta}{s}})$ and its energy measure $\nu_{\scriptscriptstyle{\langle g\rangle}}$ is given by
\begin{align*}
\nu_{\scriptscriptstyle{\langle g\rangle}}=2\sum_{\varnothing\not=B\subset I}| \nabla^B g |^2\,\mu_{\scriptscriptstyle{B}}^{\varrho,n,\replaced{\beta}{s}}.
\end{align*}
In particular, for $g \in \mathcal{D}$ holds
\begin{align*}
\nu_{\scriptscriptstyle{\langle g\rangle}}=2\sum_{\varnothing\not=B\subset I} \sum_{i \in B}  \big(\partial_i g\big)^2\,\mu_{\scriptscriptstyle{B}}^{\varrho,n,\replaced{\beta}{s}}.
\end{align*}
\end{proposition}

\begin{proof}
First let $f,g\in\mathcal{D}$. We have
\begin{multline*}
2\,\mathcal{E}(gf,g)-\mathcal{E}(g^2,f)\\
=2\sum_{\varnothing\not=B\subset E}\,\int_{E_+(B)}\sum_{i\in B}\partial_i\big( gf\big)\,\partial_i g\,d\mu_{\scriptscriptstyle{B}}^{\varrho,n,\replaced{\beta}{s}}-\sum_{\varnothing\not=B\subset I}\int_{E_+(B)}\sum_{i\in B}\partial_i\big(g^2\big)\,\partial_i f\,d\mu_{\scriptscriptstyle{B}}^{\varrho,n,\replaced{\beta}{s}}\\
=2\sum_{\varnothing\not=B\subset I}\int_{E_+(B)}\sum_{i\in B}\Big(\partial_i gf+g\partial_i f\Big)\,\partial_i g\,d\mu_{\scriptscriptstyle{B}}^{\varrho,n,\replaced{\beta}{s}}-2\sum_{\varnothing\not=B\subset I}\int_{E_+(B)}\sum_{i\in B}g\,\partial_i g\,\partial_i f\,d\mu_{\scriptscriptstyle{B}}^{\varrho,n,\replaced{\beta}{s}}\\
=2\sum_{\varnothing\not=B\subset I}\int_{E_+(B)}\sum_{i\in B}\Big(\big(\partial_i g\big)^2\,f+g\,\partial_i g\,\partial_i f\Big)\,d\mu_{\scriptscriptstyle{B}}^{\varrho,n,\replaced{\beta}{s}}-2\sum_{\varnothing\not=B\subset I}\int_{E_+(B)}\sum_{i\in B}g\,\partial_i g\,\partial_i f\,d\mu_{\scriptscriptstyle{B}}^{\varrho,n,\replaced{\beta}{s}}\\
=2\sum_{\varnothing\not=B\subset I}\int_{E_+(B)}\sum_{i\in B}\big(\partial_i g\big)^2\,f\,d\mu_{\scriptscriptstyle{B}}^{\varrho,n,\replaced{\beta}{s}}=\int_{E}f\,\,2\sum_{\varnothing\not=B\subset I}\sum_{i\in B}\big(\partial_i g\big)^2\,d\mu_{\scriptscriptstyle{B}}^{\varrho,n,\replaced{\beta}{s}}.
\end{multline*}
This shows the assertion for $g \in \mathcal{D}$, since $\mathcal{D}$ is dense in $C^0_c(E)$ with respect to $\Vert \cdot \Vert_{\text{sup}}$.\\
Now let $g\in D(\mathcal{E}) \cap C^0_c(E)$ and $f\in \mathcal{D}$. Moreover, let $(g_k)_{k\in\mathbb{N}}\subset\mathcal{D}$ such that $g_k\to g$ in $\big(D(\mathcal{E}),\Vert\cdot\Vert_{\mathcal{E}_1}\big)$. By \cite[p.123]{FOT94} it holds
\[ \Bigg| \Big( \int_E f d\nu_{\langle g \rangle} \Big)^{\frac{1}{2}} -  \Big( \int_E f d\nu_{\langle g_k \rangle} \Big)^{\frac{1}{2}} \Bigg| \leq  \Big( \int_E f d\nu_{\langle g-g_k \rangle} \Big)^{\frac{1}{2}} \leq \sqrt{2 \Vert f \Vert_{\text{sup}} \ \mathcal{E}(g-g_k,g-g_k)}. \]
Hence
\begin{align*}
\int_E f d\nu_{\langle g \rangle} = \lim_{k \rightarrow \infty} \int_E f d\nu_{\langle g_k \rangle} &=\lim_{k\to\infty}\Big(2\,\mathcal{E}\big(f g_k,g_k\big)-\mathcal{E}\big(g_k^2,f\big)\Big) \\
&=\lim_{k\to\infty} \int_{E}f\,\,2\sum_{\varnothing\not=B\subset I}\sum_{i\in B}\big(\partial_i g_k\big)^2\,d\mu_{\scriptscriptstyle{B}}^{\varrho,n,\replaced{\beta}{s}}.
\end{align*}
It \replaced{remains}{rests} to show that $g$ possesses on each set $E_+(B) \setminus \{ \varrho=0 \}$ a square-integrable weak gradient (with respect to $\mu_B^{\varrho,n,\replaced{\beta}{s}}$) and that $\nabla^B g_k \to \nabla^B g$ in $L^2(E_+(B);\mu_B^{\varrho,n,\replaced{\beta}{s}})$ as $k \rightarrow \infty$.\\
Define $G_j:=[0,j)^n \cap (E \setminus \overline{B_{\frac{1}{j}}(\{ \varrho=0\}})$ and $G_j^B := E_+(B) \cap G_j$ for $j \in \mathbb{N}$. Then each $G_j$ fulfills the assumptions of Lemma \ref{lemweakdiff} and $G_j^B \uparrow E_+(B) \setminus \{ \varrho =0\}$ as $j \replaced{\uparrow}{\rightarrow} \infty$. This yields a weak gradient $\nabla^B g$ of $g$ on each set $G_j^B$ and therefore a weak gradient in $L^1_{\text{loc}}(E_+(B) \setminus \{\varrho=0\})$. Additionally, it holds
\[ \int_{E_+(B)} |\nabla^B g|^2 \ d\mu_{\scriptscriptstyle{B}}^{\varrho,n,\replaced{\beta}{s}} \leq \liminf_{j \rightarrow \infty} \int_{E_+(B)} \mathbbm{1}_{G_j} |\nabla^B g|^2 \ d\mu_{\scriptscriptstyle{B}}^{\varrho,n,\replaced{\beta}{s}} \leq \mathcal{E}^{\scriptscriptstyle{B}}(g,g), \]
since the last inequality holds for fixed $j \in \mathbb{N}$. This shows that $\nabla^B g \in L^2(E_+(B);\mu_{\scriptscriptstyle{B}}^{\varrho,n,\replaced{\beta}{s}})$ and furthermore, applying the above inequality to $g-g_k$ finishes the proof.
\end{proof}

\begin{proposition}\label{propmeasureform}
Suppose that Condition\deleted{s} \deleted{\ref{conddensity}, \ref{condHamza} and} \ref{condweakdiff} \replaced{is}{are} satisfied. Let $f,g\in\mathcal{D}\subset D(\mathcal{E})$. Then
\begin{align*}
\mathcal{E}\big(f,g\big)=\big\langle\nu_f,g\big\rangle:=\int_{E}g\,d\nu_f
\end{align*}
with 
\begin{align*}
\nu_f:=\sum_{B\subset I}\Big(-\Delta^B f-\Big(\nabla^B f,\nabla^B \ln(\varrho)\Big)\Big)\,\lambda^{\varrho,n,\replaced{\beta}{s}}_{\scriptscriptstyle{B}}-\frac{1}{\replaced{\beta}{s}}\sum_{B\subsetneq I}\Big(\nabla^{I\setminus {B}}f,e\Big)\,\lambda^{\varrho,n,\replaced{\beta}{s}}_{\scriptscriptstyle{B}}.
\end{align*}
\end{proposition}

\begin{proof}
This representation is valid due to the integration by parts carried out in the proof of Proposition \ref{propibp}.
\end{proof}

Next we recall the definition of a \emph{positive, continuous, additive functional} (see e.g.~\cite[Appendix A.2,~A.3]{FOT94}).

\begin{definition}[\emph{additive functional}]
A family $\big(A_t\big)_{t\ge 0}$ of extended real valued functions $A_t:\Omega\to\overline{\mathbb{R}}$, with $\overline{\mathbb{R}}:=\mathbb{R}\cup\{-\infty,\infty\}$, is called \emph{additive functional} \deleted{(used not in the strict sense, no change)}(\emph{A}F in abbreviation) of $\mathbf{M}$ if it satisfies the following conditions:
\begin{enumerate}
\item[(A1)]
$A_t$ is $\mathbf{F}_t$-measurable for each $t\ge 0$.
\item[(A2)]
There exists $\Lambda\in\mathbf{F}_{\scriptscriptstyle{\infty}}$ with $\mathbf{P}^{\varrho,n,\replaced{\beta}{s}}_x(\Lambda)=1$ for all $x\in E$, $\mathbf{\Theta}_t\Lambda\subset\Lambda$ for all $t>0$ and for each $\omega\in\Lambda$, $t\mapsto A_t(\omega)$ is right continuous and has left limit on $\big[0,\infty\big)$ satisfying
\begin{enumerate}
\item[(i)]
$A_0(\omega)=0$, and 
\item[(ii)]
$A_{t+\replaced{s}{\tau}}(\omega)=A_t(\omega)+A_{\replaced{s}{\tau}}(\mathbf{\Theta}_t\omega)$ for all $t,\replaced{s}{\tau}\ge 0$.
\end{enumerate}
\end{enumerate}
The set $\Lambda$ in the above is called a \emph{defining set} for $\big(A_t\big)_{t\ge 0}$. An $\big(A_t\big)_{t\ge 0}$ is said to be \emph{finite} if $\big|A_t(\omega)\big|<\infty$ for all $t\in[0,\infty)$ and each $\omega$ in a defining set.
An $\big(A_t\big)_{t\ge 0}$ is said to be \emph{continuous} if $[0,\infty)\ni t\mapsto A_t(\omega)\in\overline{\mathbb{R}}$ is continuous for each $\omega$ in a defining set. A continuous AF $\big(A_t\big)_{t\ge 0}$ consisting of a family of $[0,\infty]$-valued functions $A_t:\Omega\to[0,\infty]$ is called a \emph{positive continuous AF} (\emph{PCAF} in abbreviation). The set of all PCAFs we denote by $\mathbf{A}^+_{c}$. Moreover, we call an AF which is also a square integrable martingale with respect to $\big(\mathbf{F}_t\big)_{t\ge 0}$ a \emph{martingale AF} (\emph{MAF} in abbreviation).  
\end{definition}

\begin{remark}\label{remPACF}
Suppose that Conditions \ref{conddensity} and \ref{condHamza} are satisfied. Let $0\le g\in C^0\big(E\big)$ and $M\in\mathcal{B}\replaced{(E)}{_{\scriptscriptstyle{E}}}$. Then $A:=(A_t)_{t\ge 0}$ with
\begin{align*}
A_t(\omega):=\int_0 ^{t}g\big(\mathbf{X}_{\replaced{s}{\tau}}(\omega)\big)\,\mathbbm{1}_{M}\big(\mathbf{X}_{\replaced{s}{\tau}}(\omega)\big)\,d\replaced{s}{\tau},\quad\omega\in \Omega,
\end{align*}
is a PCAF, i.e., $A\in\mathbf{A}^+_{c}$. If $g$ is bounded, $A$ is even finite. \added{Compare e.g.~\cite[Exam.~5.1.1]{FOT94}.}
\end{remark}
\deleted{Proposition + proof}

Given $\mathbf{M}$ and a positive measure $\mu$ on $\big(E,\mathcal{B}\replaced{(E)}{_{\scriptscriptstyle{E}}}\big)$ we define a positive measure $\mathbf{P}_\mu$ on $(\mathbf{\Omega},\mathbf{F})$ by
\begin{align*}
\mathbf{P}_\mu(\Gamma):=\int_E\mathbf{P}^{\varrho,n,\replaced{\beta}{s}}_x(\Gamma)\,d\mu(x),\quad\Gamma\in\mathbf{F}.
\end{align*}
Now we want to assign to the measures $\nu_{\scriptscriptstyle{\langle g\rangle}}$ from Proposition \ref{propenergymeasure} and $\nu_f$ from Proposition \ref{propmeasureform} the corresponding \emph{additive functionals (AFs)}. In order to do this we make use of \cite[Theo.~5.1.3]{FOT94}. 

We consider the following classes of measures.

 \begin{definition}[\emph{smooth measure}, \emph{measure of finite energy integral}]
We denote by $\mathbf{S}$ the family of \emph{smooth measures}, i.e., all positive Borel measures $\mu$ on $\mathcal{B}\replaced{(E)}{_{\scriptscriptstyle{E}}}$ such that $\mu$ charges no set of \deleted{$\mathcal{E}$-}capacity zero and there exists an increasing sequence $\big(F_k\big)_{k\in\mathbb{N}}$ of closed sets in $E$ such that $\mu\big(F_k\big)<\infty$ for all $k\in\mathbb{N}$ and $\lim_{k\to\infty}\text{cap}_{\deleted{\mathcal{E}}}\big(K\setminus F_k\big)=0$ for any compact set $K\subset E$. Here $\text{cap}_{\deleted{\mathcal{E}}}\big(S\big)$ denotes the \deleted{$\mathcal{E}$-}capacity of a set $S\subset E$.  

A positive Radon measure $\mu$ on $\mathcal{B}\replaced{(E)}{_{\scriptscriptstyle{E}}}$ is said to be of \emph{finite energy integral} if
\begin{align*}
\int_E\big|f\big|\,d\mu\le C_4\sqrt{\mathcal{E}_1\big(f,f\big)},\quad f\in D(\mathcal{E})\cap C^0_c(E),
\end{align*}
for some $C_4\in(0,\infty)$. We denote by $\mathbf{S}_{\scriptscriptstyle{0}}$ the set of all positive Radon measures of finite energy integral\deleted{s}.
\end{definition}

\begin{remark}\label{remUpot}
A positive Radon measure $\mu$ on $\mathcal{B}\replaced{(E)}{_{\scriptscriptstyle{E}}}$ is of finite energy integral if and only if there exists for each $\alpha>0$ a unique $U_\alpha\mu\in D\big(\mathcal{E}\big)$ such that
\begin{align*}
\mathcal{E}_\alpha\big(U_\alpha\mu,f\big)=\int_E f\,d\mu\quad\mbox{for all }f\in D\big(\mathcal{E}\big)\cap C^0_c(E),
\end{align*}
where $\mathcal{E}_\alpha(\cdot,\cdot):={\mathcal{E}(\cdot,\cdot)}+{\alpha\,\big(\cdot,\cdot\big)_{L^2(E;\mu_{\varrho,n,\replaced{\beta}{s}})}}$.
\end{remark}

\begin{definition}[\emph{$\alpha$-potential}]
We call $U_\alpha\mu$ from Remark \ref{remUpot} an \emph{$\alpha$-potential} and denote by $\mathbf{S}_{\scriptscriptstyle{00}}$ the set of all finite $\mu\in \mathbf{S}_{\scriptscriptstyle{0}}$ such that $\Vert U_1\mu\Vert_{\scriptscriptstyle{L^\infty(E;\mu_{\varrho,n,\replaced{\beta}{s}})}}<\infty$. 
\end{definition}

\begin{remark}\label{remS00}
Let $\mu\in\mathbf{S}_{\scriptscriptstyle{00}}$ be a finite measure and $g:E\to [0,\infty)$ measurable and bounded. \added{Applying \cite[Theo.~2.2.1]{FOT94} we obtain that }\deleted{Then}${\mu_g}:=g\mu\in\mathbf{S}_{\scriptscriptstyle{00}}$.
\end{remark}
\deleted{lemma + proof}
Let $t>0$, $\mu\in \mathbf{S}$, $A\in\mathbf{A}^+_{c}$ and $f,h:E\to[0,\infty)$ measurable. Then we consider 
\begin{align}\label{equRevuz1}
\mathbb{E}_{h \mu_{\varrho,n,\replaced{\beta}{s}}}\Big(\big(fA\big)_t\Big):=\int_{\mathbf{\Omega}}\int_{0}^tf\big(\mathbf{X}_{\replaced{s}{\tau}}\big)\,dA_{\replaced{s}{\tau}}\,d\mathbf{P}_{h\cdot\mu_{\varrho,n,\replaced{\beta}{s}}}
\end{align}
and
\begin{multline}\label{equRevuz2}
\int_0^t\left\langle f\mu,p_{\replaced{s}{\tau}} h\right\rangle\,d\replaced{s}{\tau}:=\int_0^t\int_E \left(\replaced{p_sh}{\widetilde{T_\tau h}}\right)f\,d\mu\,d\replaced{s}{\tau}\\
=\int_0^t\int_E\int_{\mathbf{\Omega}}h\big(\mathbf{X}_{\replaced{s}{\tau}}(\omega)\big)\,d\mathbf{P}^{\varrho,n,\replaced{\beta}{s}}_x(\omega)\,f(x)\,d\mu(x)\,d\replaced{s}{\tau}\replaced{.}{,}
\end{multline}
\deleted{where for $\tau\ge 0$, $p_\tau h:=\widetilde{T_\tau h}$ denotes an $\mathcal{E}$-quasi continuous version of ${T_\tau h}$.}

\begin{definition}[\emph{Revuz correspondence}]\label{defRevuz}
A measure $\mu\in \mathbf{S}$ and a AF $A\in\mathbf{A}^+_{c}$ are said to be in \emph{Revuz correspondence} if and only if equality of (\ref{equRevuz1}) and (\ref{equRevuz2}) holds for all $f,h:E\to[0,\infty)$ measurable.
\end{definition}

\begin{remark}\label{Revuzbasic}
Suppose that Condition\deleted{s} \deleted{\ref{conddensity}, \ref{condHamza} and} \ref{condweakdiff} \replaced{is}{are} satisfied. Using the symmetry of $(p_t)_{t\ge 0}$ in (\ref{equRevuz2}) one easily checks that the measure $\mu_{\varrho,n,\replaced{\beta}{s}}$ is in Revuz correspondence with the PCAF $(A_t)_{t\ge 0}:=(t)_{t\ge 0}$.
\end{remark}

\begin{remark}\label{remRevuz}
Suppose that Condition\deleted{s} \deleted{\ref{conddensity}, \ref{condHamza} and} \ref{condweakdiff} \replaced{is}{are} satisfied. Then for $B\subset I$ the positive Radon measure $\mu_{\scriptscriptstyle{B}}:=\mu^{n,\varrho}_{\scriptscriptstyle{B}}:=\varrho\,\lambda^{\scriptscriptstyle{(n)}}_{\scriptscriptstyle{B}}$ is an element of $\mathbf{S}_{\scriptscriptstyle{00}}$ and, by using Remark \ref{Revuzbasic} together with \cite[Lemma 5.1.3]{FOT94}, in Revuz correspondence with the PCAF $\big(A_t^{\scriptscriptstyle{B}}\big)_{t\ge 0}$ given by 
\begin{align*}
A^{\scriptscriptstyle{B}}_t:=A^{n,\replaced{\beta}{s},\scriptscriptstyle{B}}_t:=\frac{1}{\replaced{\beta}{s}^{n-\# B}}\int_{0}^t\mathbbm{1}_{E_+(B)}\big(\mathbf{X}_{\replaced{s}{\tau}}\big)\,d\replaced{s}{\tau}.
\end{align*}
\end{remark}
\deleted{theorem}

\deleted{adapt the proof}
\begin{remark}\label{remS00ex}
Suppose Conditions \ref{conddensity} and \ref{condHamza}. Let $\eta\in C^0(E)$ such that ${\eta}\ge 0$. Again by applying \cite[Lemma 5.1.3]{FOT94} and using Remark \ref{remS00} we obtain:
\begin{enumerate}
\item[(i)]
If $\mu\in\mathbf{S}_{\scriptscriptstyle{00}}$ and $A_t=\int_0^t g\big(\mathbf{X}_{\replaced{s}{\tau}}\big)\mathbbm{1}_M\big(\mathbf{X}_{\replaced{s}{\tau}}\big)\,d\replaced{s}{\tau}$, $t\ge 0$, as in Remark \ref{remPACF}, are in Revuz correspondence, then
\begin{align*}
{\mu_{\eta}}:={\eta}\mu\quad\text{and}\quad{A}^{\eta}_t:=\int_0^t {\eta}\big(\mathbf{X}_{\replaced{s}{\tau}}\big)\,g\big(\mathbf{X}_{\replaced{s}{\tau}}\big)\mathbbm{1}_M\big(\mathbf{X}_{\replaced{s}{\tau}}\big)\,d\replaced{s}{\tau}
\end{align*}
are in Revuz correspondence.
\item[(ii)]
If, moreover, ${\eta}$ has compact support, then ${\mu_{\eta}}\in\mathbf{S}_{\scriptscriptstyle{00}}$.  
\end{enumerate}
\end{remark}
\deleted{deleted proposition with proof}

\begin{remark}\label{remS00plus}
If $\mu_1,\mu_2\in\mathbf{S}_{\scriptscriptstyle{00}}$ with Revuz corresponding AFs $A_1$, $A_2$, respectively. Then $\mu_1+\mu_2\in\mathbf{S}_{\scriptscriptstyle{00}}$ with Revuz corresponding AF $A$ given by $A:=A_1+A_2$.
\end{remark}

\begin{theorem}\label{theorep1}
Suppose that Condition\deleted{s} \deleted{\ref{conddensity}, \ref{condHamza} and} \ref{condweakdiff} \replaced{is}{are} satisfied. Let $f\in\mathcal{D}$. Then
\begin{align}\label{equdecomp}
f\big(\mathbf{X}_t\big)-f\big(\mathbf{X}_0\big)=\mathbf{M}_t^{[f]}+\mathbf{N}_t^{[f]}\quad\mathbf{P}^{\varrho,n,\replaced{\beta}{s}}_x-\text{a.s.~for q.e. }x\in E,
\end{align}
where $\mathbf{M}_t^{[f]}$ is a MAF with quadratic variation
\begin{align*}
\left\langle \mathbf{M}^{[f]} \right\rangle_t=2\sum_{\varnothing\not=B\subset I}\int_0^t\left|\nabla^B f\big(\mathbf{X}_{\replaced{s}{\tau}}\big)\right|^2\deleted{_{\scriptscriptstyle{\text{euc}}}}\,\mathbbm{1}_{{E_+(B)}}\big(\mathbf{X}_{\replaced{s}{\tau}}\big)\,d\replaced{s}{\tau}
\end{align*}
and
\begin{multline*}
\mathbf{N}^{[f]}_t=\int_0^t\Big(\sum_{B\subset I}\Big(\Big(\Delta^B f+\big(\nabla^B f,\nabla^B\ln(\varrho)\big)\deleted{_{\scriptscriptstyle{\text{euc}}}}\Big)\big(\mathbf{X}_{\replaced{s}{\tau}}\big)\,\mathbbm{1}_{{E_+(B)}}\big(\mathbf{X}_{\replaced{s}{\tau}}\big)\Big)\\
+\Big(\sum_{B\subset I}\frac{1}{s}\big(\nabla^{I\setminus B}f,e \big)(\mathbf{X}_{\replaced{s}{\tau}}\big)\,\mathbbm{1}_{{E_+(B)}}\big(\mathbf{X}_{\replaced{s}{\tau}}\big)\Big)\,d\replaced{s}{\tau}.
\end{multline*}
\end{theorem}

\begin{remark}
Note that the decomposition (\ref{equdecomp}) is valid $\mathbf{P}^{\varrho,n,\replaced{\beta}{s}}_x-\text{a.s.~for q.e. }x\in E$. This is weaker then the statement in \cite[Theo.~5.2.5]{FOT94} where the decomposition holds $\mathbf{P}^{\varrho,n,\replaced{\beta}{s}}_x-\text{a.s.~for each }x\in E$. This is caused by the fact that in our setting we do not know if the \emph{absolute continuity condition} is fulfilled.  
\end{remark}

\begin{proof}
We have to check the assumptions of \cite[Theo.~5.2.5]{FOT94}. $f\in\mathcal{D}\subset D\big(\mathcal{E}\big)$ is clearly bounded and continuous. The measure $\nu_{\scriptscriptstyle{\langle f\rangle}}\in \mathbf{S}_{\scriptscriptstyle{00}}$ due to Proposition \ref{propenergymeasure}, Remarks \ref{remRevuz}, \ref{remS00ex}(ii) and \ref{remS00plus} applied inductively. In addition, these results yield that $\nu_{\scriptscriptstyle{\langle f\rangle}}$ is in Revuz correspondence with the PCAF
\begin{align*}
2\sum_{\varnothing\not=B\subset I}\int_0^t\sum_{i\in B}\big(\partial_if\big)^2\big(\mathbf{X}_{\replaced{s}{\tau}}\big)\,\mathbbm{1}_{{E_+(B)}}\big(\mathbf{X}_{\replaced{s}{\tau}}\big)\,d\replaced{s}{\tau}.
\end{align*}
By Proposition \ref{propmeasureform} 
\begin{align*}
\mathcal{E}\big(f,g\big)=\big\langle\nu_f,g\big\rangle=\int_{E}g\,d\nu_f
\end{align*}
with 
\begin{align*}
\nu_f=\sum_{B\subset I}\Bigg(\sum_{i\in B}\Big(-\partial^2_i f-\partial_i f\partial_i \ln(\varrho)\Big)\Bigg)\,\lambda^{\varrho,n,\replaced{\beta}{s}}_{\scriptscriptstyle{B}}-\frac{1}{\replaced{\beta}{s}}\sum_{B\subset I}\Big(\sum_{i\in I\setminus B}\partial_i f\Big)\,\lambda^{\varrho,n,\replaced{\beta}{s}}_{\scriptscriptstyle{B}}
\end{align*}
for all $f,g\in\mathcal{D}$. We can split the densities contained in $\nu_{f}$ into positive and negative part. This yields two positive Radon measures $\nu_{f}^{+}$ and $\nu_{f}^{-}$ such that $\nu_{f}=\nu_{f}^{+}-\nu_{f}^{-}$. These measures belong to $\mathbf{S}_{\scriptscriptstyle{00}}$ by Remarks \ref{remRevuz}, \ref{remS00ex} and \ref{remS00plus}. We can calculate the associated PCAFs $A^{\scriptscriptstyle{+}}$ and $A^{\scriptscriptstyle{-}}$ in the same way like in the case of $\nu_{\scriptscriptstyle{\langle f\rangle}}$. By \cite[Theo.~5.2.5]{FOT94} $\mathbf{N}^{[f]}_t=-A^{\scriptscriptstyle{+}}+A^{\scriptscriptstyle{-}}$ and we obtain that
\begin{multline*}
\mathbf{N}^{[f]}_t=\int_0^t\Big(\sum_{B\subset I}\Big(\sum_{i\in B}\Big(\partial^2_i f+\partial_i f\partial_i \ln(\varrho)\Big)\Big)\big(\mathbf{X}_{\replaced{s}{\tau}}\big)\,\mathbbm{1}_{{E_+(B)}}\big(\mathbf{X}_{\replaced{s}{\tau}}\big)\Big)\\
+\Big(\sum_{B\subset I}\Big(\frac{1}{\replaced{\beta}{s}}\sum_{i\in I\setminus B}\partial_i f\Big)(\mathbf{X}_{\replaced{s}{\tau}}\big)\,\mathbbm{1}_{{E_+(B)}}\big(\mathbf{X}_{\replaced{s}{\tau}}\big)\Big)\,d\replaced{s}{\tau}.
\end{multline*}
\end{proof}

\begin{corollary}\label{coroskoro}
Let $j\in I$. We denote by $\pi_j:\mathbb{R}^n\to\mathbb{R}$, $x\mapsto x_j$, the projection on the $j$-th coordinate. Then under the assumptions of Theorem \ref{theorep1} the coordinate processes $\big(\mathbf{X}_t^j\big)_{t\ge 0}:=\big(\pi_j(\mathbf{X}_{t})\big)_{t\ge 0}$, $1\le j\le n$, corresponding to $\mathbf{M}$ is given by
\begin{multline}\label{rep}
\mathbf{X}_{t}^j-\mathbf{X}_0^j=\sqrt{2}\,\int_0^t\mathbbm{1}_{\mathring{E}}\big(\mathbf{X}_{\replaced{s}{\tau}}\big)\,dB^j_{\replaced{s}{\tau}}+\int_{0}^{t}\partial_j\ln(\varrho)\big(\mathbf{X}_{\replaced{s}{\tau}}\big)\mathbbm{1}_{\mathring{E}}\big(\mathbf{X}_{\replaced{s}{\tau}}\big)\,d\replaced{s}{\tau}\\
+\sum_{\varnothing\not=B\subsetneq I}\left\{\begin{array}{ll}
  \sqrt{2}\,\int_0^t\mathbbm{1}_{{E_+(B)}}\big(\mathbf{X}_{\replaced{s}{\tau}}\big)\,dB^j_{\replaced{s}{\tau}}+\int_0^{t} \partial_j\ln(\varrho)\big(\mathbf{X}_{\replaced{s}{\tau}}\big)\mathbbm{1}_{{E_+(B)}}\big(\mathbf{X}_{\replaced{s}{\tau}}\big)\,d\replaced{s}{\tau}, & \text{if }j\in B\\
  \frac{1}{\replaced{\beta}{s}}\,\int_0^{t}\mathbbm{1}_{{E_+(B)}}\big(\mathbf{X}_{\replaced{s}{\tau}}\big)\,d\replaced{s}{\tau}, & \text{if }j\in I\setminus B
  \end{array}\right.\\
+\frac{1}{\replaced{\beta}{s}}\int_0^{t}\mathbbm{1}_{\{(0,\ldots,0)\}}\big(\mathbf{X}_{\replaced{s}{\tau}}\big)\,d\replaced{s}{\tau}\quad\added{\mathbf{P}^{\varrho,n,\beta}_x-\text{a.s.}\text{ for q.e.~}x\in E},
\end{multline}
where $(B^j_t)_{t\ge 0}$ is a one dimensional standard Brownian motion. Moreover, $(B^j_t)_{t\ge 0}$ and $(B^i_t)_{t\ge 0}$ are independent for $i,j\in I$ with $i\not=j$.
\end{corollary}

\begin{proof}
We consider
\begin{align*}
\pi_j^k(x):=
\left\{\begin{array}{ll}
  x_j, & \text{if }x\in [0,k+1)^n\\
  0, & \text{if }x\in [k+2,\infty)^n
  \end{array}\right., \quad 1 \le j \le n, \, k \in\mathbb{N},\text{ such that }\pi_j^k\in\mathcal{D}.
\end{align*}
Furthermore, we define
\begin{align*}
\tau_k:=\inf\big\{t \ge 0 \,|\, \mathbf{X}_t \not\in [0,k]^n\big\}, \quad k \in {\mathbb N}. 
\end{align*}
$(\tau_k)_{k\in\mathbb{N}}$ is a sequence of stopping times with $\tau_k\uparrow\infty$ as $k\to\infty$. Now using the decomposition (\ref{equdecomp}) we obtain for $k\in\mathbb{N}$ and $j\in I$ the representation
\begin{multline*}
\mathbf{X}_{t\wedge\tau_k}^j-\mathbf{X}_0^j=\pi_j^k\big(\mathbf{X}_{t\wedge\tau_k}\big)-\pi_j^k\big(\mathbf{X}_0\big)=\mathbf{M}_{t\wedge\tau_k}^{[\pi_j^k]}+\mathbf{N}_{t\wedge\tau_k}^{[\pi_j^k]}\\
=\mathbf{M}_{t\wedge\tau_k}^{[\pi^k_j]}+\int_{0}^{t\wedge\tau_k}\partial_j\ln(\varrho)\big(\mathbf{X}_{\replaced{s}{\tau}}\big)\mathbbm{1}_{\mathring{E}}\big(\mathbf{X}_{\replaced{s}{\tau}}\big)\,d\replaced{s}{\tau}\\
+\sum_{\varnothing\not=B\subsetneq I}\left\{\begin{array}{ll}
  \int_0^{t\wedge\tau_k} \partial_j\varrho\big(\mathbf{X}_{\replaced{s}{\tau}}\big)\mathbbm{1}_{{E_+(B)}}\big(\mathbf{X}_{\replaced{s}{\tau}}\big)\,d\replaced{s}{\tau}, & \text{if }j\in B\\
  \frac{1}{\replaced{\beta}{s}}\,\int_0^{t\wedge\tau_k}\mathbbm{1}_{{E_+(B)}}\big(\mathbf{X}_{\replaced{s}{\tau}}\big)\,d\replaced{s}{\tau}, & \text{if }j\in I\setminus B
  \end{array}\right.\\
+\frac{1}{\replaced{\beta}{s}}\int_0^{t\wedge\tau_k}\mathbbm{1}_{\{(0,\ldots,0)\}}\big(\mathbf{X}_{\replaced{s}{\tau}}\big)\,d\replaced{s}{\tau}\quad\added{\mathbf{P}^{\varrho,n,\beta}_x-\text{a.s.~for q.e.~}x\in E}
\end{multline*}
 Additionally we have the cross variation
\begin{align*}
\left\langle \mathbf{M}^{[\pi^k_i]},\mathbf{M}^{[\pi^k_j]}\right\rangle_{t\wedge\tau_k}=\delta_{ij}\,\left\langle \mathbf{M}^{[\pi^k_j]}\right\rangle_{t\wedge\tau_k}=\delta_{ij}\,\sum_{\varnothing\not=B\subset I}2\,\int_0^{t\wedge\tau_k}\mathbbm{1}_{{E_+(B)}}\big(\mathbf{X}_{\replaced{s}{\tau}}\big)\,d\replaced{s}{\tau}\quad \added{\mathbf{P}^{\varrho,n,\beta}_x-\text{a.s., }x\in E}.
\end{align*}

For $k\in\mathbb{N}$ large enough $\mathbf{M}_{t\wedge\tau_k}^{[\pi^k_j]}=\mathbf{M}_{t}^{[\pi_j]}$ is a continuous, local martingale and moreover, for fixed $\varnothing\not=B\subset I$ and $i,j\in B$ we have that
\begin{align*}
\left\langle \mathbf{M}^{[\pi_i]},\mathbf{M}^{[\pi_j]}\right\rangle_{t}=\delta_{ij}\,2\,\int_0^{t}\mathbbm{1}_{{E_+(B)}}\big(\mathbf{X}_{\replaced{s}{\tau}}\big)\,d\replaced{s}{\tau}=\int_0^{t}\left(\delta_{ij}\,\sqrt{2}\,\mathbbm{1}_{{E_+(B)}}\big(\mathbf{X}_{\replaced{s}{\tau}}\big)\right)^2\,d\replaced{s}{\tau}\quad\added{\mathbf{P}^{\varrho,n,\beta}_x-\text{a.s.},~x\in E}.
\end{align*}
Thus for $t\ge 0$ and $j\in B$ we obtain \added{(perhaps after enlarging the probability space)} by using \cite[Theo.~18.12]{Ka02} that
\begin{align*}
\mathbf{M}^{[\pi_j]}_t=\sqrt{2}\int_0^t\mathbbm{1}_{{E_+(B)}}\big(\mathbf{X}_{\replaced{s}{\tau}}\big)\,dB^j_{\replaced{s}{\tau}}\quad\added{\mathbf{P}^{\varrho,n,\beta}_x-}\text{a.s.}\added{, x\in E}, 
\end{align*}
where $(B^j_t)_{t\ge 0}$ is a one dimensional standard Brownian motion. Moreover, $(B^j_t)_{t\ge 0}$ and $(B^i_t)_{t\ge 0}$ are independent for $i,j\in I$ with $i\not=j$. 
\end{proof}

\section{Ergodicity and occupation time}\label{sectergo} 

\begin{definition}[\emph{part of a Dirichlet form}] Let $(\mathcal{G}, D(\mathcal{G}))$ be an arbitrary regular Dirichlet form on some locally compact separable metric space $X$, $m$ a positive Radon measure on $X$ with full topological support and $G$ an open subset of $X$. Then we define by $\mathcal{G}^G(f,g):=\mathcal{G}(f,g)$ for $f,g \in \{ f \in D(\mathcal{G}) | \ \tilde{f}=0 \text{ q.e. on } X \setminus G \}$ the \emph{part of the form $(\mathcal{G},D(\mathcal{G}))$ on $G$}, where $\tilde{f}$ denotes a quasi-continuous version of $f$. Indeed, this defines a regular Dirichlet form on $L^2(G; m)$ and for any special standard core $\mathcal{C}$ of $(\mathcal{G}, D(\mathcal{G}))$, $\mathcal{C}_G:= \{ f \in \mathcal{C} | \text{ supp}[f] \subset G \}$ is a core of $(\mathcal{G},D(\mathcal{G}))$ (see \cite[Theorem 4.4.3]{FOT94}).
\end{definition}

Throughout this section, suppose Condition\deleted{s} \deleted{\ref{conddensity}, \ref{condHamza} and} \ref{condweakdiff} \replaced{is}{are} satisfied and denote by 
\begin{align*}
\mathbf{M}:=\mathbf{M}^{\varrho,n,\replaced{\beta}{s}}:=\big(\mathbf{\Omega},\mathbf{F},(\mathbf{F}_t)_{t\ge 0},(\mathbf{X}_t)_{t\ge 0},(\mathbf{\Theta}_t)_{t\ge 0},(\mathbf{P}^{\varrho,n,\replaced{\beta}{s}}_x)_{x\in E}\big)
\end{align*}
the process constructed in Theorem \ref{theoprocess}. Furthermore, for an open subset $G$ of $E$
\begin{align*} 
\mathbf{M}^G:=\big(\mathbf{\Omega},\mathbf{F},(\mathbf{F}_t)_{t\ge 0},(\mathbf{X}^0_t)_{t\ge 0},(\mathbf{\Theta}_t)_{t\ge 0},(\mathbf{P}^{\varrho,n,\replaced{\beta}{s}}_x)_{x\in G_{\Delta}}\big)
\end{align*}
is called the part of the process $\mathbf{M}$ on $G$, where $\mathbf{X}^0_t(\omega)$ results from $\mathbf{X}_t(\omega)$ by killing the path upon leaving $G$ for $\omega \in \mathbf{\Omega}$. Here $G_{\Delta}:=G\cup\{\Delta\}$, where $\Delta$ denotes the cemetery, see e.g.~\cite[Chap. A.2]{FOT94}. By \cite[Theorem 4.4.2]{FOT94} the process $\mathbf{M}^G$ is associated to $(\mathcal{E}^G,D(\mathcal{E}^G))$. \\

In (\ref{form1}) we defined the form $\mathcal{E}_{\scriptscriptstyle{B}}$ for $\varnothing \neq B \subset I$ and functions $f,g \in \mathcal{D}$. We can extend the definition to functions in $f,g \in C_c^2(\overline{E_{\scriptscriptstyle{+}}(B)})$. Denote the closure of the pre-Dirichlet form $(\mathcal{E}_{\scriptscriptstyle{B}},C_c^2(\overline{E_{\scriptscriptstyle{+}}(B)}))$ on $L^2(\overline{E_{\scriptscriptstyle{+}}(B)};\mu_{\scriptscriptstyle{B}}^{\varrho,n,\replaced{\beta}{s}})$ by $(\mathcal{E}_{\scriptscriptstyle{B}},D(\mathcal{E}_{\scriptscriptstyle{B}}))$ and by $(T_t^{\scriptscriptstyle{B}})_{t >0}$ the corresponding semigroup. It is known that this yields for each $B$ a strongly local, recurrent, regular Dirichlet form.\\
Let $A_i$, $i \in \mathcal{I}$, be the connected components of $\tilde{E}:=E \setminus \{\varrho=0\}$ for some index set $\mathcal{I}$ and $A_i^B:= A_i \cap \overline{E_{\scriptscriptstyle{+}}(B)}$. We suppose an additional condition:

\begin{condition} \label{condcomp} $\mathcal{I}$ is finite, each $A_i$, $i \in \mathcal{I}$, is convex and the density $\varrho$ fulfills
\[ \int_{B_r (\{ \varrho =0\})} d\mu_{\varrho,n,\replaced{\beta}{s}} \leq C r^2 \text{ as } r \rightarrow 0 \]
with a constant $C < \infty$.
\end{condition}

For the following lemma we need the notion of a strongly regular Dirichlet form (see also \cite[Section 4.2]{Stu94} and \cite{Stu95}):

\begin{definition}[\emph{strong regularity}] A regular Dirichlet form $(\mathcal{G},D(\mathcal{G}))$ on $L^2(E;\mu_{\varrho,n,\replaced{\beta}{s}})$ is called \emph{strongly regular}, if the topology induced by the intrinsic metric
\[ d(x,y):= \sup \{ f(x)-f(y) | \ f \in D(\mathcal{G}) \cap C^0_c(E) \text{ with } \frac{1}{2} \nu_{\scriptscriptstyle{\langle f\rangle}} \leq \mu_{\varrho,n,\replaced{\beta}{s}} \}, \ \ x,y \in E, \]
coincides with the topology generated by the euclidean metric on $E$. Here $\frac{1}{2} \nu_{\scriptscriptstyle{\langle f\rangle}} \leq \mu_{\varrho,n,\replaced{\beta}{s}}$ means that the energy measure of $f$ is absolutely continuous w.r.t.~$\mu_{\varrho,n,\replaced{\beta}{s}}$ and its Radon-Nikodym derivative is almost everywhere less or equal than two.
\end{definition}

\begin{lemma}\label{lemcap} 
\begin{enumerate} \item[(i)] $\{ \varrho =0 \}$ is of \deleted{$\mathcal{E}$-}capacity zero and $\{ \varrho =0 \} \cap \overline{E_{\scriptscriptstyle{+}}(B)}$ is of \deleted{$\mathcal{E}_B$-}capacity zero for every $\varnothing \neq B \subset I$.
\item[(ii)] $A_i$ is open and $T_t$-invariant for every $i \in \mathcal{I}$ .
\item[(iii)] $A_i^B$ is open in $\overline{E_{\scriptscriptstyle{+}}(B)}$ and $T_t^B$-invariant for every $i \in \mathcal{I}$ and $\varnothing \neq B \subset I$.
\end{enumerate}
\end{lemma}

\begin{proof} (i) We only show the first statement for the Dirichlet form $(\mathcal{E},D(\mathcal{E}))$. The second statement follows for the same reasons. By \cite[Theorem 3]{Stu95} and Condition \ref{condcomp} it is enough to show strong regularity of $(\mathcal{E},D(\mathcal{E}))$. Let $f \in D(\mathcal{E}) \cap C^0_c(E)$. Then the energy measure of $f$ has the form 
\[ \frac{1}{2} \nu_{\scriptscriptstyle{\langle f\rangle}} = \sum_{\varnothing \neq B \subset I} | \nabla^B f |^2 \mu_B^{\varrho,n,\replaced{\beta}{s}} = \Big( \sum_{\varnothing \neq B \subset I} | \nabla^B f |^2 \mathbbm{1}_{E_{\scriptscriptstyle{+}}(B)} \Big)\ \mu_{\varrho,n,\replaced{\beta}{s}} \]
by Proposition \ref{propenergymeasure}.
Thus, 
\begin{align*} d(x,y)&= \sup \{ f(x)-f(y) | \ f \in D(\mathcal{E}) \cap C_c(E) \text{ with } \sum_{\varnothing \neq B \subset I} | \nabla^B f |^2 \mathbbm{1}_{E_{\scriptscriptstyle{+}}(B)} \leq 1 \text{ a.e. on } E \} \\
&= \sup \{ f(x)-f(y) | \ f \in D(\mathcal{E}) \cap C_c(E) \text{ with } | \nabla^B f |^2 \leq 1 \text{ a.e. on } E_{\scriptscriptstyle{+}}(B), \ \varnothing \neq B \subset I \}
\end{align*}
for $x,y \in E$. Since $E$ is convex, we have by the fundamental theorem of calculus
\[ \added{\vert x-y\vert} = d(x,y).\]
This proves the assertion.\\
(ii) Clearly, each $A_i$ is open in $E$, since $\varrho$ is continuous. In order to show $T_t$-invariance, it is sufficient to prove that $A_i$ is quasi open and quasi closed simultaneously by \cite[Corollary 4.6.3]{FOT94}.
Since each open set is quasi open, it is left to show that $A_i$ is quasi closed or equivalently that $E \setminus A_i$ is quasi open. Thus, let $\varepsilon >0$.  Since $\{ \varrho =0 \}$ is of \deleted{$\mathcal{E}$-}capacity zero by (i), there exists an open set $B$ containing $\{ \varrho =0 \}$ with $\text{cap}(B) < \varepsilon$. The set $G:= \cup_{j \neq i} A_j \cup B$ is open, contains $E \setminus A_i$ and it holds
\[ \text{cap}(G \setminus (E \setminus A_i)) \leq \text{cap}(B) < \varepsilon. \]
Hence $E\setminus A_i$ is quasi closed. Thus $A_i$ is $T_t$-invariant.\\
(iii) follows by the same arguments.
\end{proof}

\begin{remark} \label{reminv}
\begin{enumerate}
\item[(i)] Due to \cite[Lemma 4.6.3]{FOT94}, $T_t$-invariance of $A_i$ implies that there exists a properly exceptional set $N_i$ such that $A_i \setminus N_i$ is $\mathbf{M}$-invariant in the sense that 
\[ \mathbf{P}^{\varrho,n,\replaced{\beta}{s}}_x(\mathbf{X}_t \in (A_i \setminus N_i) \text{ for all } t \geq 0)=1 \text{ for all } x \in A_i \setminus N_i. \]
\item[(ii)] Fix some set $A_{i}^B$ and define $G_k :=\{ x \in A_{i}^B | \ d(x, \{ \varrho =0\} \cap \overline{E_{\scriptscriptstyle{+}}(B)}) > \frac{1}{k} \} \cap [0,k)^n$. This yields a sequence of bounded open subsets of $A_{i}^B$ increasing to $A_{i}^B$. Since $\varrho \in C^1(E)$ by Condition \ref{condweakdiff}, it follows that $\gamma_k := \text{ess inf}_{x \in G_k} \ \varrho >0$, $k=1,2,\dots,$ (with respect to the measure $\lambda^{(n)}_B$).
\item[(iii)] By a similar argument, $L^p$-norms on $K$ with respect to the measures $\mu_B^{\varrho,n,\replaced{\beta}{s}}$ and $\lambda_B^{(n)}$ respectively are equivalent for some compact set $K$ \deleted{properly} contained in some $A_i^B$.
\item[(iv)] In the case that $\varrho(x) >0$ for all $x \in E$, $\tilde{E}=E$ is already connected. Moreover, in Condition \ref{condcomp} instead of assuming convexity it suffices to require that $\mathcal{I}$ is finite and the intersection of some $A_i$ with $E_{\scriptscriptstyle{+}}(B)$ is either empty or connected.
\end{enumerate}
\end{remark}

\begin{theorem}\label{spendtime}
Let $i \in \mathcal{I}$. For all $f \in L^1(A_i;\mu_{\varrho,n,\replaced{\beta}{s}})$ holds
\begin{align}
\lim_{t \rightarrow \infty} \frac{1}{t} \int_0^t f(\mathbf{X}_{\replaced{s}{\tau}}) d\replaced{s}{\tau}= \frac{\int_{A_i} f d\mu_{\varrho,n,\replaced{\beta}{s}}}{\mu_{\varrho,n,\replaced{\beta}{s}}(A_i)}
\end{align}
$\mathbf{P}_x^{\varrho,n,\replaced{\beta}{s}}$-a.s. for q.e. $x \in A_i$.
\end{theorem}
\begin{proof}
Fix $i \in \mathcal{I}$. Due to \cite[Theorem 4.7.3(iii)]{FOT94}, the definition of $\mathbf{M}^{A_i}$ and Remark \ref{reminv} (i) it is sufficient to show that $(\mathcal{E}^{A_i},D(\mathcal{E}^{A_i}))$ is irreducible recurrent. \deleted{In order to deduce recurrence of $(\mathcal{E},D(\mathcal{E}))$, it is enough to show that there exists a sequence $(f_k)_{k \in \mathbb{N}} \subset D(\mathcal{E})$ such that $\lim_{k \rightarrow \infty} f_k =1$ $\mu_{\varrho,n,\replaced{\beta}{s}}$-a.e. and $\lim_{k \rightarrow \infty} \mathcal{E}(f_k,f_k)=0$ by} \cite[Theorem 1.6.3]{FOT94}. \added{Recurrence}\deleted{This} has already been \replaced{shown}{done} in \deleted{the proof of} Proposition \ref{propDirichlet}. In particular, \added{we have}\deleted{it follows} that $\mathbbm{1}_E \in D(\mathcal{E})$ and $\mathcal{E}(\mathbbm{1}_E,\mathbbm{1_E})=0$. Since $A_i$ is $T_t$-invariant, we have $\mathbbm{1}_{A_i}=\mathbbm{1}_{A_i} \mathbbm{1}_E \in D(\mathcal{E})$ and 
\[ 0=\mathcal{E}(\mathbbm{1}_E,\mathbbm{1}_E)= \mathcal{E}(\mathbbm{1}_{A_i},\mathbbm{1}_{A_i}) + \mathcal{E}(\mathbbm{1}_{E \setminus A_i},\mathbbm{1}_{E \setminus A_i}) \]
by \cite[Theorem 1.6.1]{FOT94}. Hence, $\mathbbm{1}_{A_i} \in D(\mathcal{E}^{A_i})$ and $\mathcal{E}^{A_i}(\mathbbm{1}_{A_i},\mathbbm{1}_{A_i})=0$ which implies recurrence of $(\mathcal{E}^{A_i},D(\mathcal{E}^{A_i}))$.
Taking into account that the considered form is recurrent, irreducibility is equivalent to the condition that every $f \in D(\mathcal{E}^{A_i})$ with $\mathcal{E}^{A_i}(f,f)=0$ is $\mu_{\varrho,n,\replaced{\beta}{s}}$-a.e. constant (on $A_i$) by \cite[Theorem 2.1.11]{ChFu11}.  Denote by $(\mathcal{E}^{i}_{\scriptscriptstyle{B}},D(\mathcal{E}^{i}_{\scriptscriptstyle{B}}))$ the part of the form $(\mathcal{E}_{\scriptscriptstyle{B}},D(\mathcal{E}_{\scriptscriptstyle{B}}))$ on $A_{i}^{\scriptscriptstyle{B}}$. This is the closure of $(\mathcal{E}_B,C_c^2(A_{i}^{\scriptscriptstyle{B}}))$ by \cite[Theorem 4.4.3]{FOT94} and thus, it is irreducible. Indeed, the closure of the pre-Dirichlet form
\[  \sum_{i \in B} \int_{A_{i}^B} \partial_i f \partial_i g d\lambda_{\scriptscriptstyle{B}}^{(n)}, \ \ f,g \in C_c^2(\overline{A_{i}^{\scriptscriptstyle{B}}}), \]
on $L^2(\overline{A_{i}^{\scriptscriptstyle{B}}};\lambda_{\scriptscriptstyle{B}}^{(n)})$ yields reflecting Brownian motion on $\overline{A_i^{\scriptscriptstyle{B}}}$ which is irreducible (see \cite[p.128]{ChFu11}) and hence the closure of the form defined for functions in $C_c^2(A_i^{\scriptscriptstyle{B}})$ on $L^2(A_i^{\scriptscriptstyle{B}};\lambda_{\scriptscriptstyle{B}}^{(n)})$ is also irreducible in view of \cite[Theorem 2.1.11]{ChFu11}. Hence, it follows by \cite[Corollary 4.6.4]{FOT94} and Remark \ref{reminv} (ii) that $(\mathcal{E}^{i}_{\scriptscriptstyle{B}},D(\mathcal{E}^{i}_{\scriptscriptstyle{B}}))$ is irreducible.\\
Let $f \in D(\mathcal{E}^{A_i})$ and choose a seqeunce $(f_k)_{k \in \mathbb{N}}$ in $C_c^2(A_i)$ such that $f_k \rightarrow f$ with respect to $\sqrt{\mathcal{E}^{A_i}_1}$. Then the restriction to $\overline{E_{\scriptscriptstyle{+}}(B)}$ is by definition $\mathcal{E}_{\scriptscriptstyle{B}}$-Cauchy and converges to the restriction of $f$ in $L^2(\overline{E_{\scriptscriptstyle{+}}(B)} ;\mu_{\scriptscriptstyle{B}}^{\varrho,n,\replaced{\beta}{s}})$. Therefore, the convergence holds also in $D(\mathcal{E}_{\scriptscriptstyle{B}})$ and
\[ \mathcal{E}^{A_i}(f,f)=\mathcal{E}(f,f)= \lim_{k \rightarrow \infty} \mathcal{E}(f_k,f_k)=\lim_{k \rightarrow \infty} \sum_{\substack{\varnothing \neq B \subset I \\ A_i^B \neq \varnothing}} \mathcal{E}_{\scriptscriptstyle{B}}(f_k,f_k) = \sum_{\substack{\varnothing \neq B \subset I \\ A_i^B \neq \varnothing}} \mathcal{E}_{\scriptscriptstyle{B}}(f,f) \]
by definition. By $T_t^{\scriptscriptstyle{B}}$-invariance
\[ \mathcal{E}^{A_i}(f,f) = \sum_{\substack{\varnothing \neq B \subset I \\ A_i^B \neq \varnothing}
} \mathcal{E}^{i}_{\scriptscriptstyle{B}}(\mathbbm{1}_{A_{i}^B} f, \mathbbm{1}_{A_{i}^B} f).\]
Therefore, $\mathcal{E}^{A_i}(f,f)=0$ implies $\mathcal{E}^{i}_{\scriptscriptstyle{B}}(\mathbbm{1}_{A_{i}^B} f, \mathbbm{1}_{A_{i}^B} f)=0$ for each summand and hence, $f=c^{i}_B$ $\mu_B^{\varrho,n,\replaced{\beta}{s}}$-a.e. on $A_{i}^{\scriptscriptstyle{B}}$ for some constant $c^{i}_{\scriptscriptstyle{B}}$ by \cite[Theorem 2.1.11]{ChFu11}. Thus, we can conclude \[ f =\sum_{A_i^{\scriptscriptstyle{B}}\not=\varnothing} c^{i}_{\scriptscriptstyle{B}} \mathbbm{1}_{A_{i}^B}.\]
It \replaced{remains}{rests} to show that there exists a constant $c$ such that $c^{i}_{\scriptscriptstyle{B}}=c$ for all $B$. Let $\varnothing \neq B \subset I$ with $A_i^{\scriptscriptstyle{B}} \neq \varnothing$ be arbitrary and $l \in B$ such that $A_i$ also intersects $E_{\scriptscriptstyle{+}}(B \setminus \{l\})$ 
. We show that $c^i_{\scriptscriptstyle{B}} =c^i_{\scriptscriptstyle{{B \setminus \{ l \}}}}$ by contradiction. Then the assertion follows by applying this result successively. If $c^i_{\scriptscriptstyle{B}} \neq c^i_{\scriptscriptstyle{{B \setminus \{l \}}}}$, we can assume that $c^i_{\scriptscriptstyle{B}}=0$ and $c^i_{\scriptscriptstyle{{B \setminus \{ l \}}}}=1$, since $\frac{c^i_B \mathbbm{1}_{A_i} -f}{c^i_B - c^i_{B \setminus \{l\}}} \in D(\mathcal{E}^{A_i})$. Fix a point $z \in A_i^{\scriptscriptstyle{B \setminus \{l\}}}$. Then, by construction \added{there} exists a (bounded) neighborhood $U$ of $z$ in $E$ such that its closure is contained in $A_i$. Choose a $C^{\infty}$-cutoff function $\eta$ defined on $E$ which is constantly one near $z$  and has support \deleted{properly} contained in $U$. Then it is easy to see that $\eta f \in D(\mathcal{E}^{A_i})$ and $(\eta f_k)_{k \in \mathbb{N}}$ is an approximation for $\eta f$ whenever $(f_k)_{k \in \mathbb{N}}$ is a sequence of $C_c^2(A_i)$-functions which approximates $f$ in $D(\mathcal{E}^{A_i})$. Moreover, $\mathcal{E}_{\scriptscriptstyle{B}}(\eta f_k,\eta f_k) \rightarrow 0$ as $k \rightarrow \infty$. By construction we have
\[ \eta f_k(x) = \eta f_k(x) - \eta f_k(x + C e_l) = - \int_0^C \partial_l  (\eta f_k)(x+t e_l) dt, \]
where $x \in U \cap \overline{E_{\scriptscriptstyle{+}}(B \setminus \{l \})}$ and $C >0$ is chosen such that $x+ C e_l \in U \setminus \text{supp}(\eta)$. Hence,
\[ |\eta f_k(x)| \leq \int_0^{\infty} | \partial_l(\eta f_k)(x+ t e_l)| dt. \]
This implies 
\[ \int_{A_i^{B \setminus \{l \})}} |\eta f_k| \ d\lambda_{B \setminus \{l\}}^{(n)} \leq \int_{A_i^B} | \partial_j (\eta f_k)| \ d\lambda_{\scriptscriptstyle{B}}^{\scriptscriptstyle{(n)}}. \]
Since we can restrict our considerations to the closure of $U$ by construction, we have equivalence of norms by Remark \ref{reminv} (iii) and hence, the left hand side converges to a positive constant, whereas the right hand side converges to zero. This is a contradiction and thus $c^i_{\scriptscriptstyle{B}}=c$ for some constant $c$, all $B$ and $i$.
\end{proof}

By the preceding ergodic theorem it follows immediately by choosing $f$ as the indicator function of the boundary that the occupation time of the process $\mathbf{M}$ on the boundary increases asymptotically linear whenever the process starts in a component which possesses a boundary part with $\mu_{\varrho,n,\replaced{\beta}{s}}$ positive measure. In particular, the process spends in this case $\mathbf{P}_x^{\varrho,n,\replaced{\beta}{s}}$-a.s. a positive amount of time at the boundary (with respect to the Lebesgue measure).

\begin{corollary} \label{corospendtime}
For all measurable $\Gamma\subset \partial E = \dot{\bigcup}_{B\subsetneq I}{E}_{\scriptscriptstyle{+}}(B)$ and all $i \in \mathcal{I}$ holds
\begin{align} \label{bdryoccup}
\lim_{t \rightarrow \infty} \frac{1}{t} \int_0^{t}\mathbbm{1}_{\Gamma}\big(\mathbf{X}_{\replaced{s}{\tau}}\big)\,d\replaced{s}{\tau}= \frac{\mu_{\varrho,n,\replaced{\beta}{s}}\big(\Gamma \cap A_i \big)}{\mu_{\varrho,n,\replaced{\beta}{s}}(A_i)}
\end{align}
$\mathbf{P}_x^{\varrho,n,\replaced{\beta}{s}}$-a.s. for q.e. $x \in A_i$. In particular, under the condition that $\mu_{\varrho,n,\replaced{\beta}{s}}(\Gamma \cap A_i) >0$ for q.e.~$x\in A_i$ and $\mathbf{P}_x^{\varrho,n,\replaced{\beta}{s}}$-a.a.~$\omega\in\mathbf{\Omega}$  there exists $T(\omega,x)\in [0,\infty)$ and $c(\omega,x)\in(0,\infty)$ such that
\begin{align}\label{proportional}
\int_0^{t}\mathbbm{1}_{\Gamma}\big(\mathbf{X}_{\replaced{s}{\tau}}(\omega)\big)\,d\replaced{s}{\tau} \ge t\, c(\omega,x) \quad\text{for all}\quad t\ge T(\omega,x).
\end{align}
\end{corollary}

\begin{corollary}\label{names} 
Let $\varrho > 0$ \added{pointwisely}, $j\in I$ and $B\not=I$.
 Then
\begin{align} \label{jzero}
\lim_{t \rightarrow \infty} \frac{1}{t} \int_0^t \mathbbm{1}_{\{0\}}(X^j_{\replaced{s}{\tau}}) d\replaced{s}{\tau} = \frac{\mu_{\varrho,n,\replaced{\beta}{s}}(\{x_j=0\})}{\mu_{\varrho,n,\replaced{\beta}{s}}(E)} >0 
\end{align}
and
\begin{align*}
\lim_{t \rightarrow \infty} \frac{1}{t} \int_0^t \mathbbm{1}_{E_+(B)}(X_{\replaced{s}{\tau}}) d\replaced{s}{\tau} = \frac{\mu_{\varrho,n,\replaced{\beta}{s}}(E_+(B))}{\mu_{\varrho,n,\replaced{\beta}{s}}(E)} >0 
\end{align*}
$\mathbf{P}_x^{\varrho,n,\replaced{\beta}{s}}$-a.s. for q.e. $x \in E$ and (\ref{proportional}) holds. Moreover, the right hand side of (\ref{jzero}) is increasing in $\replaced{\beta}{s}$, converges to $1$ as $\replaced{\beta}{s} \rightarrow \infty$ and converges to $0$ as $\replaced{\beta}{s} \rightarrow 0$.

\end{corollary}

\begin{proof} The first statement follows directly from (\ref{bdryoccup}). In order to proof the latter assertion note that 
\[  \frac{\mu_{\varrho,n,\replaced{\beta}{s}}(\{ x_j=0 \})}{\mu_{\varrho,n,\replaced{\beta}{s}}(E)} = 1- \frac{\mu_{\varrho,n,\replaced{\beta}{s}}( \{ x_j >0 \})}{\mu_{\varrho,n,\replaced{\beta}{s}}(E)} \]
and
\begin{align} \label{decreq}
\frac{\mu_{\varrho,n,\replaced{\beta}{s}}( \{ x_j >0 \})}{\mu_{\varrho,n,\replaced{\beta}{s}}(E)}= \frac{\sum_{\substack{B \subset I \\ j \in B}} \int_{E_+(B) } \replaced{\beta}{s}^{n-\#B} \varrho d\lambda_{\scriptscriptstyle{B}}^{\scriptscriptstyle{(n)}}}{\sum_{B \subset I} \int_{E_+(B) } \replaced{\beta}{s}^{n-\#B} \varrho d\lambda_{\scriptscriptstyle{B}}^{\scriptscriptstyle{(n)}}}=\frac{\sum_{i=0}^n \replaced{\beta}{s}^i a_i}{\sum_{i=0}^n \replaced{\beta}{s}^i b_i},
\end{align}
where 
\begin{align*} &a_i:= \sum_{\substack{B \subset I \\ \#B=n-i, \ j \in B}} \int_{E_+(B) } \varrho d\lambda_{\scriptscriptstyle{B}}^{\scriptscriptstyle{(n)}} \ \text{ for } i=0,\dots, n-1, \\ &b_i:= \sum_{\substack{B \subset I \\ \#B=n-i}} \int_{E_+(B) } \varrho d\lambda_{\scriptscriptstyle{B}}^{\scriptscriptstyle{(n)}} \ \text{ for } i=0,\dots, n
\end{align*}
and $a_n:=0$. It holds $0 < a_i < b_i$ for $i=1,\dots,n-1$, $0=a_n<b_n$ and $0 < a_0=b_0$. Hence, (\ref{decreq}) is decreasing in $\replaced{\beta}{s}$, converges to $0$ as $\replaced{\beta}{s} \rightarrow \infty$ and converges to $1$ as $\replaced{\beta}{s} \rightarrow 0$.
\end{proof}

\section{Application to the dynamical wetting model in $(d+1)$-dimension}

Let $d\in\mathbb{N}$ and $D_{\scriptscriptstyle{d}}:=(0,1]^d\subset\mathbb{R}^d$. For $N\in\mathbb{N}$ we define $D_{\scriptscriptstyle{d,N}}:=ND_{\scriptscriptstyle{d}}\cap\mathbb{Z}^d$, where $ND_{\scriptscriptstyle{d}}:=\big\{N\theta\,\big|\,\theta\in D_{\scriptscriptstyle{d}}\big\}$. Here $N$ stands for the scaling parameter. The discretized set $D_{d,\scriptscriptstyle{N}}$ is a \emph{microscopic} correspondence to the \emph{macroscopic} domain $D_{\scriptscriptstyle{d}}$ and given by $D_{\scriptscriptstyle{d,N}}=\big\{1,2,\ldots,N\big\}^d$. The boundary $\partial D_{\scriptscriptstyle{d,N}}$ of $D_{\scriptscriptstyle{d,N}}$ is defined by $\partial D_{\scriptscriptstyle{d,N}}:=\big\{\added{x\in\mathbb{Z}^d\setminus D_{\scriptscriptstyle{d,N}}}\,\big|\,|x-y|=1\mbox{ for some }y\in D_{\scriptscriptstyle{d,N}}\big\}$ and the closure $\overline{D_{d,\scriptscriptstyle{N}}}$ of $D_{\scriptscriptstyle{d,N}}$ is \added{defined} by $\overline{D_{d,\scriptscriptstyle{N}}}:=D_{\scriptscriptstyle{d,N}}\cup\partial D_{d,\scriptscriptstyle{N}}$. Hence $\overline{D_{\scriptscriptstyle{d,N}}}=\big\{0,1,2,\ldots,N+1\big\}^d$. For fixed $N\in\mathbb{N}$ we consider the \emph{space of interfaces}
\begin{align*}
\overline{\Omega^+_{\scriptscriptstyle{d,N}}}:=[0,\infty)^{{D_{\scriptscriptstyle{d,N}}}}:=\Big\{\phi:{D_{\scriptscriptstyle{d,N}}}\to[0,\infty)\Big\}
\end{align*}
on ${D_{\scriptscriptstyle{d,N}}}$. Note that $\deleted{\phi_x:=}\phi(x)$ describes the height of an interface $\phi\in\overline{\Omega^+_{\scriptscriptstyle{d,N}}}$ at position $x\in {D_{\scriptscriptstyle{d,N}}}$ with respect to the reference hyperplane ${D_{\scriptscriptstyle{d,N}}}$. Therefore, $\replaced{\phi(x)}{\phi_x}$, $x\in {D_{\scriptscriptstyle{d,N}}}$, is also called \emph{height variable}. We extend $\phi\in\overline{\Omega_{\scriptscriptstyle{d,N}}^+}$ to the boundary $\partial D_{\scriptscriptstyle{d,N}}$ by setting $\replaced{\phi(x)}{\phi_x}=0$ for all $x\in\partial D_{\scriptscriptstyle{d,N}}$. The restriction for the functions $\phi$ to take values in $[0,\infty)\subset\mathbb{R}$ reflects the fact that a hard wall is settled at height level $0$ of the interface. 

The potential energy of an interface $\phi\in\overline{\Omega_{\scriptscriptstyle{d,N}}^+}$ is given by a \emph{Hamiltonian} with \emph{zero boundary condition}, i.e.,
\begin{align}\label{equHam}
\overline{\Omega_{\scriptscriptstyle{d,N}}^+}\ni\phi\mapsto H^{\scriptscriptstyle{V}}_{\scriptscriptstyle{d,N}}(\phi):=\frac{1}{2}\sum_{\stackunder{|x-y|=1}{x,y\in{{\overline{D_{\scriptscriptstyle{d,N}}}}}}}V\big(\replaced{\phi(x)-\phi(y)}{\phi_x-\phi_y}\big)\in\mathbb{R},
\end{align}
where the pair interaction potential $V$ fulfills Condition \ref{condpotential} below.
\begin{condition}\label{condpotential}
The potential $V:\mathbb{R}\to[-b,\infty)$, $b\in[0,\infty)$, is continuously differentiable and symmetric, i.e., $V(-r)=V(r)$ for all $r\in\mathbb{R}$ and moreover, $\kappa:=\int_{\mathbb{R}}\exp\big(-V(r)\big)\,dr<\infty$.
\end{condition}

A natural distribution on the space of interfaces $\Big(\overline{\Omega_{\scriptscriptstyle{d,N}}^+},\mathcal{B}\big(\overline{\Omega_{\scriptscriptstyle{d,N}}^+}\big)\Big)$ is given by the probability measure $\mu_{\scriptscriptstyle{d,N}}^{\scriptscriptstyle{V,\replaced{\beta}{s}}}$ defined by
\begin{align}\label{repmeasure}
d\mu_{\scriptscriptstyle{d,N}}^{\scriptscriptstyle{V,\replaced{\beta}{s}}}(\phi)=\frac{1}{Z_{\scriptscriptstyle{d,N}}^{\scriptscriptstyle{V,\replaced{\beta}{s}}}}\exp\Big(-H^{\scriptscriptstyle{V}}_{\scriptscriptstyle{d,N}}(\phi)\Big)\,\prod_{x\in {D_{\scriptscriptstyle{d,N}}}}\Big(\replaced{\beta}{s}\,d\replaced{\delta_0(x)}{\delta_0^x}+d\replaced{\phi_+(x)}{\phi_+^x}\Big),\quad\phi\in\overline{\Omega_{\scriptscriptstyle{d,N}}^+},
\end{align}
with pair interaction potential $V$ under Condition \ref{condpotential} and normalizing constant $Z_{\scriptscriptstyle{d,N}}^{\scriptscriptstyle{V,\replaced{\beta}{s}}}$. \added{Here $\prod_{x\in {D_{\scriptscriptstyle{d,N}}}}\Big(\,d\delta_0(x)+d\phi_+(x)\Big)$ denotes the product measure on $[0,\infty)^{N^d}$, where $d\phi_+(x)$ is the Lebesgue measure on $\big([0,\infty),\mathcal{B}\big([0,\infty)\big)\big)$ and $\delta_{\scriptscriptstyle{0}}(x)$ denotes the Dirac measure on $\big([0,\infty),\mathcal{B}\big([0,\infty)\big)\big)$ at $0$ for $x\in D_{\scriptscriptstyle{d,N}}$.} $\mu_{\scriptscriptstyle{d,N}}^{\scriptscriptstyle{V,\replaced{\beta}{s}}}$ is a \emph{finite volume Gibbs measure} conditioned on $[0,\infty)^{{D_{\scriptscriptstyle{d,N}}}}$.
The corresponding space of square integrable functions we denote by $L^2\Big(\overline{\Omega_{\scriptscriptstyle{d,N}}^+};\mu_{\scriptscriptstyle{d,N}}^{\scriptscriptstyle{V,\replaced{\beta}{s}}}\Big)$. Next we define the probability density
\begin{align*}
\varrho(\phi):=\varrho^{\scriptscriptstyle{V,\replaced{\beta}{s}}}_{\scriptscriptstyle{d,N}}(\phi):=\frac{1}{Z_{\scriptscriptstyle{d,N}}^{\scriptscriptstyle{V,\replaced{\beta}{s}}}}\exp\Big(-H^{\scriptscriptstyle{V}}_{\scriptscriptstyle{d,N}}(\phi)\Big),\quad\phi\in\overline{\Omega_{\scriptscriptstyle{d,N}}^+}.
\end{align*}
Hence we can rewrite (\ref{repmeasure}) as
\begin{multline*}
d\mu_{\scriptscriptstyle{N^d,\replaced{\beta}{s},\varrho}}:=d\mu_{\scriptscriptstyle{d,N}}^{\scriptscriptstyle{V,\replaced{\beta}{s}}}=\varrho\,\prod_{x\in {D_{\scriptscriptstyle{d,N}}}}\Big(\replaced{\beta}{s}\,d\replaced{\delta_0(x)}{\delta_0^x}+d\replaced{\phi_+(x)}{\phi_+^x}\Big)\\
=\varrho\,\sum_{B\subset D_{\scriptscriptstyle{d,N}}}\replaced{\beta}{s}^{N^d-\#B}\left(\prod_{x\in B}d\replaced{\phi_+(x)}{\phi^x_+}\prod_{y\in D_{\scriptscriptstyle{d,N}}\setminus B}d\replaced{\delta_0(y)}{\delta^y_{0}}\right)
=\varrho\,\sum_{B\subset D_{\scriptscriptstyle{d,N}}}d\lambda^{N^d,\replaced{\beta}{s}}_{\scriptscriptstyle{B}}
=\varrho\,dm_{N^d,\replaced{\beta}{s}},\quad\phi\in\overline{\Omega_{\scriptscriptstyle{d,N}}^+}.
\end{multline*}

\begin{condition}\label{cond2}
$\mathbb{V}'(x,\cdot)\in L^2\big(\overline{\Omega_{\scriptscriptstyle{d,N}}^+};\mu_{\scriptscriptstyle{N^d,\replaced{\beta}{s},\varrho}}\big)$ for all $x\in {D_{\scriptscriptstyle{d,N}}}$, where
\begin{align*}
\overline{\Omega_{\scriptscriptstyle{d,N}}^+}\ni\phi\mapsto \mathbb{V}'(x,\phi):=\sum_{\stackunder{|x-y|=1}{y\in{{\overline{D_{\scriptscriptstyle{d,N}}}}}}}V'\big(\replaced{\phi(x)-\phi(y)}{\phi_x-\phi_y}\big)\in\mathbb{R}.
\end{align*}
\end{condition}

\begin{remark}\label{remcondapp}
Condition \ref{condpotential} guarantees that $V(0)\in[-b,\infty)$, hence flat interfaces are natural elements in the space of interfaces $\overline{\Omega_{\scriptscriptstyle{d,N}}^+}$, i.e., occur with positive probability, see (\ref{repmeasure}). Furthermore, Conditions \ref{condpotential} and \ref{cond2} imply Conditions \deleted{\ref{condHamza},} \ref{condweakdiff} and \ref{condcomp} (see also Remark \ref{remcondequi}).
\end{remark}

\begin{remark}
\added{
In \cite{Fu05} the authors assume that the potential $V$ is twice
continuously differentiable, symmetric and strictly convex, i.e., it exist
some constants $c_{-},c_{+} >0$ such that $c_{-} \leq V^{\prime \prime}(r)
\leq c_{+}$ for all $r \in \mathbb{R}$. This implies that $\kappa:=
\int_{\mathbb{R}} \exp(-V(r)) dr < \infty$. In particular, the potentials
investigated in \cite{Fu05} obviously fulfill Condition \ref{condpotential}.
In addition, Condition \ref{cond2} is also satisfied. Indeed, in the case
$d=N=1$ with $\phi:=\phi(1)$ it holds by integration by parts}
\begin{align*}
0 &\leq \int_{[0,\infty)} \mathbb{V}^{\prime}(1,\phi)^2~ \exp(-2 V(\phi))
\big(\beta d\delta_0 + d\phi_{+}\big)\\
&= \lim_{b \rightarrow \infty} \int_0^b
\big(-2V^{\prime}(\phi)\big)\big(-2V^{\prime}(\phi)\big)~ \exp(-2 V(\phi))~
d\phi_{+} \\
&= \lim_{b \rightarrow \infty} -2 V^{\prime} (b) \exp(-2V(b)) +
\int_{[0,\infty)} 2 V^{\prime \prime}(\phi) \exp(-2 V(\phi)) ~ d\phi_{+} \\
&\leq 2 c_{+} \int_{[0,\infty)} \exp(-2 V(\phi)) ~ d\phi_{+} < \infty,
\end{align*}
\added{since $V^{\prime}$ is non-decreasing and $V^{\prime}(0)=0$.
Similar, but more lengthy, calculations show that this result is valid for higher dimensions and larger numbers of height variables. Therefore, the class of admissible potentials in our construction includes the one considered in \cite{Fu05} for the dynamical wetting model.}
\end{remark}

For each $\phi\in\overline{\Omega_{\scriptscriptstyle{d,N}}^+}$ we denote by
\begin{align*}
D^{\scriptscriptstyle{\text{dry}}}_{\scriptscriptstyle{d,N}}(\phi):=\big\{x\in D_{\scriptscriptstyle{d,N}}\,\big|\replaced{\phi(x)}{\phi_x}=0\big\}\quad\mbox{and}\quad D^{\scriptscriptstyle{\text{wet}}}_{\scriptscriptstyle{d,N}}(\phi):=\big\{x\in D_{\scriptscriptstyle{d,N}}\,\big|\replaced{\phi(x)}{\phi_x}>0\big\},
\end{align*}
\emph{dry regions} and \emph{wet regions} associated with the interface $\phi$, respectively, and define for $A,B\subset D_{\scriptscriptstyle{d,N}}$,
\begin{align*}
\Omega_{\scriptscriptstyle{d,N,A}}^{+\scriptscriptstyle{\text{,dry}}}:=\Big\{\phi\in\overline{\Omega_{\scriptscriptstyle{d,N}}^+}\,\Big|\,D^{\scriptscriptstyle{\text{dry}}}_{\scriptscriptstyle{d,N}}(\phi)=A\Big\}\quad\mbox{and}\quad\Omega_{\scriptscriptstyle{d,N,B}}^{+\scriptscriptstyle{\text{,wet}}}:=\Big\{\phi\in\overline{\Omega_{\scriptscriptstyle{d,N}}^+}\,\Big|\,D^{\scriptscriptstyle{\text{wet}}}_{\scriptscriptstyle{d,N}}(\phi)=B\Big\},
\end{align*}
respectively. 
\begin{remark}
The following decomposition of the state space is valid:
\begin{align*}
\overline{\Omega_{\scriptscriptstyle{d,N}}^+}=\dot{\bigcup}_{A\subset D_{\scriptscriptstyle{d,N}}}\Omega_{\scriptscriptstyle{d,N,A}}^{+\scriptscriptstyle{\text{,dry}}}=\dot{\bigcup}_{B\subset D_{\scriptscriptstyle{d,N}}}\Omega_{\scriptscriptstyle{d,N,B}}^{+\scriptscriptstyle{\text{,wet}}}.
\end{align*}
\end{remark}
Therefore, $\mu_{\scriptscriptstyle{N^d,\replaced{\beta}{s},\varrho}}=\sum_{B\subset D_{\scriptscriptstyle{d,N}}}\mu^{\scriptscriptstyle{N^d,\replaced{\beta}{s},\varrho}}_{\scriptscriptstyle{B}}$ with $\mu^{\scriptscriptstyle{N^d,\replaced{\beta}{s},\varrho}}_{\scriptscriptstyle{B}}:=\mu_{\scriptscriptstyle{N^d,\replaced{\beta}{s},\varrho}}\Big|_{\mathcal{B}_{\Omega_{\scriptscriptstyle{d,N,B}}^{+\scriptscriptstyle{\text{,wet}}}}}$.

\begin{theorem}\label{theosumdiriapp}
Let $d,N\in\mathbb{N}$. For $\replaced{\beta}{s}\in(0,\infty)$ we have that under Conditions \ref{condpotential} and \ref{cond2} 
\begin{align}\label{formapp}
\mathcal{E}^{N^d,\replaced{\beta}{s},\varrho}\big(F,G\big):=\sum_{\varnothing\not=B\subset D_{\scriptscriptstyle{d,N}}}\mathcal{E}^{N^d,\replaced{\beta}{s},\varrho}_{\scriptscriptstyle{B}}\big(F,G\big),\quad F,G\in \mathcal{D}=C_c^2\big(\overline{\Omega_{\scriptscriptstyle{d,N}}^+}\big)
\end{align}
with 
\begin{align*}
\mathcal{E}^{N^d,\replaced{\beta}{s},\varrho}_{\scriptscriptstyle{B}}\big(F,G\big):=\sum_{x\in B}\int_{\Omega_{\scriptscriptstyle{d,N,B}}^{+\scriptscriptstyle{\text{,wet}}}}\partial_xF\,\partial_x G\,d\mu^{\scriptscriptstyle{N^d,\replaced{\beta}{s},\varrho}}_{\scriptscriptstyle{B}},\quad\varnothing\not=B\subset D_{\scriptscriptstyle{d,N}},
\end{align*}
is a densely defined, positive definite, symmetric bilinear form, which is closable on\\ $L^2\big(\overline{\Omega_{\scriptscriptstyle{d,N}}^+};\mu_{\scriptscriptstyle{N^d,\replaced{\beta}{s},\varrho}}\big)$. Its closure $\big(\mathcal{E}^{N^d,\replaced{\beta}{s},\varrho},D(\mathcal{E}^{N^d,\replaced{\beta}{s},\varrho})\big)$ is a \added{recurrent, hence in particular conservative}, strongly local, strongly regular, symmetric Dirichlet form on $L^2\big(\overline{\Omega_{\scriptscriptstyle{d,N}}^+};\mu_{\scriptscriptstyle{N^d,\replaced{\beta}{s},\varrho}}\big)$. 
\end{theorem}

\begin{remark}
Note that for functions in $\mathcal{D}$, $l\in\{1,2\}$ and $x\in D_{\scriptscriptstyle{d,N}}$ we denote by $\partial_x^l$ the partial derivative of order $l$ with respect to the variable $\replaced{\phi(x)}{\phi_x}$. In particular, $\partial_x:=\partial_x^1$.
\end{remark}

\begin{proof}[Proof of Theorem \ref{theosumdiriapp}]
Use Remark \ref{remcondapp} and apply Theorem \ref{theosumdiri}. For strong regularity see the proof of Lemma \ref{lemcap}(i).
\end{proof}

For $F\in \mathcal{D}:=C_c^2\big(\overline{\Omega_{\scriptscriptstyle{d,N}}^+}\big)$ and $B \subset D_{\scriptscriptstyle{d,N}}$ we define
\begin{multline*}
\mathcal{L}^BF:={\mathcal{L}}^{d,N,\varrho,B} F:=\sum_{x \in B} \big( \partial^2_x F+ \partial_x F\,\partial_x (\ln\varrho) \big)+ \sum_{x \in D_{\scriptscriptstyle{d,N}} \backslash B} \frac{1}{\replaced{\beta}{s}} \partial_x F\\
 = \Delta^B F + \big( \nabla^B  F, \nabla^B \ln \varrho \big) + \frac{1}{\replaced{\beta}{s}} (\nabla^{D_{\scriptscriptstyle{d,N}} \backslash B} F,e),
\end{multline*}
and
\begin{align*} \mathcal{L}F := \sum_{B \subset D_{\scriptscriptstyle{d,N}}} \mathbbm{1}_{\Omega_{\scriptscriptstyle{d,N,B}}^{+\scriptscriptstyle{\text{,wet}}}} \mathcal{L}^B F,
\end{align*}
where $\nabla^B F:=\sum_{x \in B} \partial_x F\,e_{c(x)}$ with some $c:D_{\scriptscriptstyle{d,N}}\to\{1,\ldots,N^d\}$ bijective and $\replaced{\{e_1, \dots, e_{N^d}\}}{(e_k)_{k\in\{1,\ldots, N^d\}}}$ being the canonical basis in $\mathbb{R}^{N^d}$. Moreover, $\Delta^B F:=\sum_{x \in B} \partial_x^2 F$ for $F \in \mathcal{D}$, $B \subset D_{\scriptscriptstyle{d,N}}$ and $e$ is a vector of length $N^d$ containing only ones.

\begin{proposition}\label{propibpapp}
Suppose Conditions \ref{condpotential} and \ref{cond2} to be satisfied. Then we have the representation $\mathcal{E}^{N^d,\replaced{\beta}{s},\varrho}\big(F,G\big)=\Big(-\mathcal{L} F,G\Big)_{L^2(\overline{\Omega_{\scriptscriptstyle{d,N}}^+};\mu_{N^d,\replaced{\beta}{s},\varrho})}$ for $F, G\in\mathcal{D}$.
\end{proposition}

\begin{proof}
Use Remark \ref{remcondapp} and apply Proposition \ref{propibp}.
\end{proof}

\begin{remark}
Let $\mathcal{L}_1^B:=\Delta^B  + \big( \nabla^B  , \nabla^B \ln \varrho \big)$ and $\mathcal{L}_2^B :=(\nabla^B,e)$. Using this notation we can express $\mathcal{L}$ in the form
\begin{align*} \mathcal{L}F&= \sum_{B \subset D_{\scriptscriptstyle{d,N}}} \mathbbm{1}_{\Omega_{\scriptscriptstyle{d,N,B}}^{+\scriptscriptstyle{\text{,wet}}}}  (\mathcal{L}_1^B F + \frac{1}{\replaced{\beta}{s}} \mathcal{L}_2^{D_{\scriptscriptstyle{d,N}} \backslash B} F) \\
&= \added{\mathbbm{1}_{\Omega_{\scriptscriptstyle{d,N,I}}^{+\scriptscriptstyle{\text{,wet}}}}} \mathcal{L}_1^{D_{\scriptscriptstyle{d,N}}} F + \sum_{B \subsetneq D_{\scriptscriptstyle{d,N}}} \mathbbm{1}_{\Omega_{\scriptscriptstyle{d,N,B}}^{+\scriptscriptstyle{\text{,wet}}}} (-\mathcal{L}_1^{D_{\scriptscriptstyle{d,N}} \backslash B}F + \frac{1}{\replaced{\beta}{s}} \mathcal{L}_2^{D_{\scriptscriptstyle{d,N}} \backslash B}F),\quad F\in\mathcal{D}.
\end{align*}
The interpretation of $\mathcal{L}$ is that on $\Omega_{\scriptscriptstyle{d,N,B}}^{+\scriptscriptstyle{\text{,wet}}}$ the operator $\mathcal{L}_1^B$ describes the dynamics of the height variables $\replaced{\phi(x)}{\phi_x}$, $x\in B$, by means of a diffusive and a drift term whereas the operator $\frac{1}{\replaced{\beta}{s}} \mathcal{L}_2^{D_{\scriptscriptstyle{d,N}} \backslash B}$ forces the remaining height variables $\phi_x$, $x \in D_{\scriptscriptstyle{d,N}} \backslash B$ with constant drift $\frac{1}{\replaced{\beta}{s}}$ back to positive height. The operator $-\mathcal{L}_1^B + \frac{1}{\replaced{\beta}{s}} \mathcal{L}_2^B$ for $B \neq \varnothing$ is called a \textit{Wentzell type boundary operator}. The associated Cauchy problem can be formulated in the form
\begin{align}
\label{PDE2}
\left\{
\begin{array}{l l}
& \frac{\partial}{\partial t} U_t(\phi) = \Delta U_t(\phi) + \big(\nabla U_t(\phi), \nabla (\ln \varrho)(\phi)\big), \quad \quad t>0, \ \phi \in \overline{\Omega_{\scriptscriptstyle{d,N}}^+},  \\
& \partial_x^2 U_t(\phi) + \partial_x U_t(\phi) \partial_x (\ln \varrho)(\phi)  - \frac{1}{\replaced{\beta}{s}} \partial_x U_t(\phi)=0, \quad t >0, x \in D_{\scriptscriptstyle{d,N}},  \ \phi \in \overline{\Omega_{\scriptscriptstyle{d,N}}^+} \cap \{\phi_x=0\}, \\
& U_0(\phi)=F(\phi)
\end{array}
\right.
\end{align}
The second line of (\ref{PDE2}) is called \textit{Wentzell boundary condition (for the $x$-th height variable)}.
\end{remark}

\begin{theorem}\label{theoprocessapp}
Suppose that Conditions \ref{condpotential} and \ref{cond2} are satisfied. Then there exists a conservative diffusion process (i.e.~a strong Markov process with continuous sample paths and infinite life time)
\begin{align*}
\mathbf{M}^{N^d,\replaced{\beta}{s},\varrho}=\left(\mathbf{\Omega},\mathbf{F},(\mathbf{F}_t)_{t\ge 0},(\replaced{\boldsymbol{\phi}_t}{\mathbf{X}_t})_{t\ge 0},(\mathbf{\Theta}_t)_{t\ge 0},(\mathbf{P}^{N^d,\replaced{\beta}{s},\varrho}_{\phi})_{\phi\in \overline{\Omega_{\scriptscriptstyle{d,N}}^+}}\right)
\end{align*}
with state space $\overline{\Omega_{\scriptscriptstyle{d,N}}^+}$ which is \deleted{properly} associated with $\big(\mathcal{E}^{N^d,\replaced{\beta}{s},\varrho},D(\mathcal{E}^{N^d,\replaced{\beta}{s},\varrho})\big)$.
$\mathbf{M}^{N^d,\replaced{\beta}{s},\varrho}$ is up to $\mu_{N^d,\replaced{\beta}{s},\varrho}$-equivalence unique. In particular, $\mathbf{M}^{N^d,\replaced{\beta}{s},\varrho}$ is $\mu_{N^d,\replaced{\beta}{s},\varrho}$-symmetric and has $\mu_{N^d,\replaced{\beta}{s},\varrho}$ as invariant\added{ and reversible} measure.
\end{theorem}
\begin{proof}
Use Remark \ref{remcondapp} and apply Theorem \ref{theoprocess}.
\end{proof}

\begin{theorem}\label{theomartingaleapp}
The diffusion process $\mathbf{M}^{N^d,\replaced{\beta}{s},\varrho}$ from Theorem \ref{theoprocessapp} is up to $\mu_{N^d,\replaced{\beta}{s},\varrho}$-equivalence the unique diffusion process having $\mu_{N^d,\replaced{\beta}{s},\varrho}$ as symmetrizing measure and solving the martingale problem for $\big(\mathcal{H}^{N^d,{\replaced{\beta}{s}},\varrho},D(\mathcal{H}^{N^d,\replaced{\beta}{s},\varrho})\big)$, i.e., for all $G\in D(\mathcal{H}^{N^d,\replaced{\beta}{s},\varrho})$
\begin{align*}
\widetilde{G}(\replaced{\boldsymbol{\phi}_t}{\mathbf{X}_t})-\widetilde{G}(\replaced{\boldsymbol{\phi}_0}{\mathbf{X}_0})+\int_0^t \Big(\mathcal{H}^{N^d,{\replaced{\beta}{s}},\varrho}G\Big)(\replaced{\boldsymbol{\phi}_s}{\mathbf{X}_{\tau}})\,d\replaced{s}{\tau},\quad t\ge 0,
\end{align*}
is an $\mathbf{F}_t$-martingale under $\mathbf{P}^{N^d,\replaced{\beta}{s},\varrho}_\phi$ for \deleted{$\mathcal{E}^{N^d,\replaced{s}{\beta},\varrho}$-}quasi all $\phi\in \overline{\Omega_{\scriptscriptstyle{d,N}}^+}$.
\end{theorem}
\begin{proof}
Use Remark \ref{remcondapp} and apply Theorem \ref{theomartingale}.
\end{proof}

\begin{corollary}\label{coroskoroapp}
Suppose that Conditions \ref{condpotential} and \ref{cond2} are satisfied. Let $x\in D_{\scriptscriptstyle{d,N}}$. We denote by $\pi_x:\overline{\Omega_{\scriptscriptstyle{d,N}}^+}\to[0,\infty)$, $\phi\mapsto \replaced{\phi(x)}{\phi_x}$, the projection on the $x$-th coordinate. The coordinate processes $\big(\replaced{\boldsymbol{\phi}_t(x)}{\mathbf{X}_t^{x}}\big)_{t\ge 0}:=\big(\pi_x(\replaced{\boldsymbol{\phi}_t}{\mathbf{X}_{t}})\big)_{t\ge 0}$ corresponding to $\mathbf{M}^{N^d,\replaced{\beta}{s},\varrho}$ is given by
\begin{multline}\label{repsolapp}
\replaced{\boldsymbol{\phi}_t(x)-\boldsymbol{\phi}_0(x)}{\mathbf{X}_{t}^{x}-\mathbf{X}_0^{x}}=\sqrt{2}\,\int_0^t\mathbbm{1}_{\Omega_{\scriptscriptstyle{d,N}}^+}\big(\replaced{\boldsymbol{\phi}_s}{\mathbf{X}_{\tau}}\big)\,d\replaced{B_s(x)}{B^x_{\tau}}-\int_{0}^{t}\mathbb{V}'\big(x,\replaced{\boldsymbol{\phi}_s}{\mathbf{X}_{\tau}}\big)\mathbbm{1}_{\Omega_{\scriptscriptstyle{d,N}}^+}\big(\replaced{\boldsymbol{\phi}_s}{\mathbf{X}_{\tau}}\big)\,d\replaced{s}{\tau}\\
+\sum_{\varnothing\not=B\subsetneq D_{\scriptscriptstyle{d,N}}}\left\{\begin{array}{ll}
  \sqrt{2}\,\int_0^t\mathbbm{1}_{\Omega_{\scriptscriptstyle{d,N,B}}^{+\scriptscriptstyle{\text{,wet}}}}\big(\replaced{\boldsymbol{\phi}_s}{\mathbf{X}_{\tau}}\big)\,d\replaced{B_s(x)}{B^x_{\tau}}-\int_0^{t} \mathbb{V}'\big(x,\replaced{\boldsymbol{\phi}_s}{\mathbf{X}_{\tau}}\big)\mathbbm{1}_{\Omega_{\scriptscriptstyle{d,N,B}}^{+\scriptscriptstyle{\text{,wet}}}}\big(\replaced{\boldsymbol{\phi}_s}{\mathbf{X}_{\tau}}\big)\,d\replaced{s}{\tau}, & \text{if }x\in B\\
  \frac{1}{\replaced{\beta}{s}}\,\int_0^{t}\mathbbm{1}_{\Omega_{\scriptscriptstyle{d,N,B}}^{+\scriptscriptstyle{\text{,wet}}}}\big(\replaced{\boldsymbol{\phi}_s}{\mathbf{X}_{\tau}}\big)\,d\replaced{s}{\tau}, & \text{if }x\in D_{\scriptscriptstyle{d,N}}\setminus B
  \end{array}\right.\\
+\frac{1}{\replaced{\beta}{s}}\int_0^{t}\mathbbm{1}_{\{(0,\ldots,0)\}}\big(\replaced{\boldsymbol{\phi}_s}{\mathbf{X}_{\tau}}\big)\,d\replaced{s}{\tau},
\end{multline}
where $(\replaced{B_t(x)}{B^x_t})_{t\ge 0}$\added{, $x\in D_{\scriptscriptstyle{d,N}}$,} \replaced{are}{is a} one dimensional independent standard Brownian motion\added{s} and
\begin{align*}
\mathbb{V}'(x,\phi):=\sum_{\stackunder{|x-y|=1}{y\in{{\overline{D_{\scriptscriptstyle{d,N}}}}}}}V'\big(\replaced{\phi(x)-\phi(y)}{\phi_x-\phi_y}\big),\quad\phi\in \overline{\Omega_{\scriptscriptstyle{d,N}}^+},
\end{align*}
with pair interaction potential $V$.
\end{corollary}

\begin{proof}
Use Remark \ref{remcondapp} and apply Corollary \ref{coroskoro}.
\end{proof}

\begin{remark}
(\ref{repsolapp}) provides a weak solution to (\ref{sde}) for \deleted{$\mathcal{E}^{N^d,\replaced{\beta}{s},\varrho}$-}quasi every starting point in $\overline{\Omega_{\scriptscriptstyle{d,N}}^+}$, even for boundary points.
\end{remark}

\begin{theorem}\label{theoergoapp}
Suppose that Conditions \ref{condpotential} and \ref{cond2} are satisfied. For all $F \in L^1\big(\overline{\Omega_{\scriptscriptstyle{d,N}}^+};\mu_{\scriptscriptstyle{N^d,\replaced{\beta}{s},\varrho}}\big)$ it holds that
\begin{align*}
\lim_{t \rightarrow \infty} \frac{1}{t} \int_0^t F(\replaced{\boldsymbol{\phi}_s}{\mathbf{X}_{\tau}}) d\replaced{s}{\tau}=\int_{\Omega_{\scriptscriptstyle{d,N}}^+} F d\mu_{\scriptscriptstyle{N^d,\replaced{\beta}{s},\varrho}}
\end{align*}
$\mathbf{P}^{\scriptscriptstyle{N^d,\replaced{\beta}{s},\varrho}}_\phi$-a.s. for q.e. $\phi \in \overline{\Omega_{\scriptscriptstyle{d,N}}^+}$.
\end{theorem}

\begin{proof}
Use Remark \ref{remcondapp} and apply Theorem \ref{spendtime}.
\end{proof}

\begin{corollary} \label{corospendtimeapp}
Under the Conditions of Theorem \ref{theoergoapp} we have that for all measurable $\Gamma\subset \partial \overline{\Omega_{\scriptscriptstyle{d,N}}^+}=\dot{\bigcup}_{B\subsetneq D_{\scriptscriptstyle{d,N}}}\Omega_{\scriptscriptstyle{d,N,B}}^{+\scriptscriptstyle{\text{,wet}}}$ it holds that
\begin{align*}
\lim_{t \rightarrow \infty} \frac{1}{t} \int_0^{t}\mathbbm{1}_{\Gamma}\big(\replaced{\boldsymbol{\phi}_s}{\mathbf{X}_{\tau}}\big)\,d\replaced{s}{\tau}= \mu_{\scriptscriptstyle{N^d,\replaced{\beta}{s},\varrho}}\big(\Gamma \big)
\end{align*}
$\mathbf{P}^{\scriptscriptstyle{N^d,\replaced{\beta}{s},\varrho}}_\phi$-a.s. for q.e. $\phi \in \overline{\Omega_{\scriptscriptstyle{d,N}}^+}$. In particular, under the condition that $\mu_{\scriptscriptstyle{N^d,\replaced{\beta}{s},\varrho}}(\Gamma) >0$ for q.e.~$\phi\in \overline{\Omega_{\scriptscriptstyle{d,N}}^+}$ and $\mathbf{P}^{\scriptscriptstyle{N^d,\replaced{\beta}{s},\varrho}}_\phi$-a.a.~$\omega\in\mathbf{\Omega}$  there exists $T(\omega,\phi)\in [0,\infty)$ and $c(\omega,\phi)\in(0,\infty)$ such that
\begin{align}\label{proportionalapp}
\int_0^{t}\mathbbm{1}_{\Gamma}\big(\replaced{\boldsymbol{\phi}_s(\omega)}{\mathbf{X}_{\tau}(\omega)}\big)\,d\replaced{s}{\tau} \ge t\, c(\omega,\phi) \quad\text{for all}\quad t\ge T(\omega,\phi).
\end{align}
\end{corollary}

\begin{corollary}\label{namesapp} 
Let $x\in D_{\scriptscriptstyle{d,N}}$ and $B\not=D_{\scriptscriptstyle{d,N}}$.
Then under the Conditions of Theorem \ref{theoergoapp} we have that
\begin{align} \label{jzeroapp}
\lim_{t \rightarrow \infty} \frac{1}{t} \int_0^t \mathbbm{1}_{\{0\}}(\replaced{\boldsymbol{\phi}_s(x)}{X^x_{\tau}}) d\replaced{s}{\tau} = \mu_{\scriptscriptstyle{N^d,\replaced{\beta}{s},\varrho}}\big(\{\replaced{\phi(x)}{\phi_x}=0\}\big)>0 
\end{align}
and
\begin{align*}
\lim_{t \rightarrow \infty} \frac{1}{t} \int_0^t \mathbbm{1}_{\Omega_{\scriptscriptstyle{d,N,B}}^{+\scriptscriptstyle{\text{,wet}}}}(\replaced{\boldsymbol{\phi}_s}{X_{\tau}}) d\replaced{s}{\tau} = \mu_{\scriptscriptstyle{N^d,\replaced{\beta}{s},\varrho}}\big(\Omega_{\scriptscriptstyle{d,N,B}}^{+\scriptscriptstyle{\text{,wet}}}\big) >0 
\end{align*}
$\mathbf{P}^{\scriptscriptstyle{N^d,\replaced{\beta}{s},\varrho}}_\phi$-a.s. for q.e. $\phi \in \overline{\Omega_{\scriptscriptstyle{d,N}}^+}$ and (\ref{proportionalapp}) holds. Moreover, the right hand side of (\ref{jzeroapp}) is increasing in $\replaced{\beta}{s}$, converges to $1$ as $\replaced{\beta}{s} \rightarrow \infty$ and converges to $0$ as $\replaced{\beta}{s} \rightarrow 0$.
\end{corollary}

\begin{proof}
Use Remark \ref{remcondapp} and apply Corollary \ref{names}.
\end{proof}

\begin{remark}
Corollary \ref{namesapp} justifies that \replaced{$\beta$}{$s$} is called \emph{strength of pinning}.
\end{remark}

\subsection*{Acknowledgment}
We thank Benedikt Heinrich, Tobias Kuna, Michael R\"ockner and Heinrich von Weizs\"acker for discussions and helpful comments. \added{Moreover, we thank an anonymous referee for helpful comments improving the readability of the paper.} Financial support through the DFG project GR 1809/8-1  is
gratefully acknowledged.

\end{document}